\documentclass{amsart}
\usepackage{amsmath}
\usepackage{amsfonts}
\usepackage{graphics}
\usepackage{epsfig}
\usepackage{amssymb}
\usepackage{amscd}
\usepackage[all]{xy}
\usepackage{latexsym}
\usepackage{graphicx}
\usepackage{multirow}
\usepackage{geometry}

\newtheorem{thm}{Theorem}[section]

\newtheorem{lem}[thm]{Lemma}
\newtheorem{prop}[thm]{Proposition}

\theoremstyle{definition}
\newtheorem{Def}[thm]{Definition}
\newtheorem{rem}[thm]{Remark}
\newtheorem*{ack}{Acknowledgement}

\numberwithin{equation}{section}
\numberwithin{figure}{section}

\def\Hom{{\text{\rm{Hom}}}}

\def\trace{{\text{\rm{trace}}}}
\def\tr{{\text{\rm{tr}}}}
\def\rchi{{\hbox{\raise1.5pt\hbox{$\chi$}}}}

\def\isom{\cong}
\def\tensor{\otimes}

\def\a{\alpha}

\def\lam{\lambda}

\def\gam{\gamma}

\def\Jac{{\text{\rm{Jac}}}}
\def\Spec{{\text{\rm{Spec}}}}

\newcommand{\bea}{\begin{eqnarray}}
\newcommand{\eea}{\end{eqnarray}}
\newcommand{\be}{\begin{equation}}
\newcommand{\ee}{\end{equation}}

\newcommand{\Mbar}{{\overline{\mathcal{M}}}}

\newcommand{\bC}{{\mathbb{C}}}

\newcommand{\bZ}{{\mathbb{Z}}}

\newcommand{\cD}{{\mathcal{D}}}
\newcommand{\cE}{{\mathcal{E}}}

\newcommand{\cH}{{\mathcal{H}}}

\newcommand{\cL}{{\mathcal{L}}}
\newcommand{\cO}{{\mathcal{O}}}

\newcommand{\cU}{{\mathcal{U}}}

\newcommand{\la}{{\langle}}
\newcommand{\ra}{{\rangle}}
\newcommand{\half}{{\frac{1}{2}}}

\newcommand{\rar}{\rightarrow}
\newcommand{\lrar}{\longrightarrow}

\begin{document}
\large
\setcounter{section}{0}

\title[Quantum curves for Hitchin fibrations and 
Eynard-Orantin theory]
{Quantum curves for Hitchin fibrations and 
the Eynard-Orantin theory}

\author[O.\ Dumitrescu]{Olivia Dumitrescu}
\address{
Institut f\"ur Algebraische Geometrie\\
Fakult\"at f\"ur Mathematik und Physik\\
Leibniz Universit\"at Hannover\\
Welfengarten 1\\
30167 Hannover, Germany}
\email{dumitrescu@math.uni-hannover.de}

\author[M.\ Mulase]{Motohico Mulase}
\address{
Department of Mathematics\\
University of California\\
Davis, CA 95616--8633, U.S.A.}
\email{mulase@math.ucdavis.edu}

\begin{abstract}
We generalize the topological recursion of 
Eynard-Orantin \cite{CEO, EO1}  to 
the family of
spectral curves of Hitchin fibrations.
A \emph{spectral curve}
in the topological recursion, which is
defined to be a complex plane curve,
 is replaced with a
generic curve in the cotangent bundle $T^*C$ of 
an arbitrary smooth base curve $C$.  
We then prove that these spectral curves
are quantizable, using the new  formalism.
More precisely, we construct
the canonical generators of the 
formal $\hbar$-deformation family of 
$D$-modules over an arbitrary 
projective algebraic curve $C$
of genus greater than $1$, from the geometry
of a prescribed family of smooth  
Hitchin spectral curves associated with 
the $SL(2,\bC)$-character variety of the 
fundamental group $\pi_1(C)$.
We show that the semi-classical limit through the 
WKB approximation of these $\hbar$-deformed
$D$-modules recovers the initial 
family of Hitchin spectral curves. 
\end{abstract}

\thanks{The first author is a member of 
the Simion Stoilow Institute of Mathematics of the 
Romanian Academy}

\subjclass[2010]{Primary: 14H15, 14H60, 14H81;
Secondary: 34E20, 81T45}

\keywords{Quantum curve; Hitchin fibration;
family of spectral curves; Higgs field;
topological recursion; WKB approximation}

\maketitle

\allowdisplaybreaks

\tableofcontents

\section{Introduction and the main results}
\label{sect:intro}

A \textbf{quantum curve} 
\cite{ADKMV, EB, DFM, DHS, DHSV, 
 EM, GS, Hollands}
is a magical object. It conjecturally captures 
information of  quantum topological invariants
in an effective and compact manner.
Mathematically, it
is a $D$-module 
defined on a formal family 
of complex holomorphic curves $C[[\hbar]]$ 
that \emph{quantizes} a \textbf{spectral curve} 
$\Sigma$.
It is automatically holonomic, and its 
\textbf{semi-classical limit}
through the \textbf{WKB approximation}
induces a 
holomorphic \emph{Lagrangian immersion} 
\begin{equation}
\label{eq:spectral}
\begin{CD}
\iota :\Sigma @>>> T^*C
\\
&&@VV{\pi}V
\\
&&C
\end{CD}
\end{equation}
with respect to the natural symplectic 
structure of the cotangent bundle $T^*C$.
It is also closely related to 
 an \emph{oper} of \cite{BD1,Frenkel},
 a $\lam$-\emph{connection} 
 of Deligne (see for example, \cite{A}),
 a \emph{quantum characteristic polynomial}
in the theory of integrable models
in statistical mechanics  
\cite{CT,T},  a \emph{Cherednik algebra}
\cite{EMa}, and the differential equation
appearing in the context of determining
the \emph{Nekrasov partition function}
\cite{N} through the \emph{AGT correspondence}
\cite{AGT,BMT,Gai}.

We note  that not every morphism of curves
$\Sigma\lrar C$ is  quantizable. Clearly we
need a Lagrangian immersion for the WKB 
method to work. Therefore, 
it is natural to ask what type of conditions
we need for the existence of quantization.

The purpose of this paper is to show
that the spectral curves
associated with 
$SL(2,\bC)$-\textbf{Hitchin fibrations}
\cite{H1,H2}
are quantizable, by concretely constructing
a canonical generator 
(which is related to the \emph{conformal 
block} in the context of  the AGT correspondence)
 of 
$\hbar$-deformed $D$-modules
on an arbitrary smooth projective 
algebraic curve $C$ of genus $g(C)\ge 2$.
For this  construction we first generalize 
the \emph{topological recursion}
mechanism proposed in \cite{EO1}, 
which is strictly restricted  to the case 
of $C=\bC$ or $C=\bC^*$, to what we
call the
\textbf{Eynard-Orantin theory}, making it
 applicable to the 
spectral curves of (\ref{eq:spectral})
with an arbitrary base curve $C$. We show 
that this new formalism allows us to 
construct the desired 
quantization of $\Sigma$.

Since our work connects many different 
developments  that took 
place in a vast array of mathematics
in recent years,
we present each component that forms our
work in this introductory section.

\subsection{Generalization of the 
topological recursion of Eynard and Orantin}
The Eynard-Orantin theory that we propose in 
this paper
stems out of various physics literature, including
\cite{EB, BKMP,CEO,DV2007, EO1,M2}. 
The key ingredient
in both Hitchin fibrations and the Eynard-Orantin
theory is the notion of spectral curves.
By generalizing the original topological recursion of 
\cite{EO1},
we shall show that these spectral 
curves  are exactly the same object.
As a consequence of this identification, we obtain
a purely geometric interpretation of 
\emph{what the topological recursion  does}.
More precisely, we construct a quantum curve
 when the
spectral curve (\ref{eq:spectral}) is non-singular and  
 $\pi:\Sigma \lrar C$ is 
a  ramified double-sheeted covering.
In this particular mechanism, the Eynard-Orantin
theory that we propose solves the all-order
\textbf{Wentzel-Kramers-Brillouin 
(WKB) approximation} 
(see for example, \cite{BO}). 
The mechanism works as follows.
\begin{itemize}
\item The spectral curves of Hitchin fibrations
are quantized by the WKB method.
\item The Eynard-Orantin theory gives a solution
to the \emph{all-order, exact}
 WKB approximation from the geometry 
of spectral curves as the initial condition.
\item Along the 
 branched points of $\pi: \Sigma\lrar C$, 
 the WKB method
does not work. Around these points, 
asymptotically, the
Eynard-Orantin theory gives the expected
\textbf{Airy function} solution 
 \cite{EB}, 
in the same
way as it appears in \cite{A}. 
This is because the local behavior of 
$\pi$ around a branched point is the double-sheeted
covering of a formal disc by another formal disc,
ramified at the origin.
The Airy function here
is representing the 
Witten-Kontsevich
theory of the 
cotangent $\psi$-class 
intersection numbers on $\Mbar_{g,n}$
 \cite{K1992,W1991} (see also \cite{DMSS}).
\end{itemize}
We note that the relation between Hitchin
systems and the WKB method is extensively
studied in
Gaiotto-Moore-Neitzke  \cite{GMN}
 and their subsequenct
papers.

The first 
formulation of the 
topological recursion in 
\cite{CEO,EO1}  assumes
 that the base curve $C$ is
always  the complex line $\bC$. 
A modification is proposed for 
the case of $C=\bC^*$ in 
\cite{BKMP, BM}.
Our current 
work  provides a  generalization of these
theories to \emph{compact} base curve $C$.
The original case is just a restriction of our
picture onto an affine piece of $C$. From this
point of view, 
we develop a \emph{global} topological
recursion, utilizing the full global
structure of the starting 
spectral curve.  The main technical difficulty
of the theory that we 
overcome in 
this paper is our global calculation of the residue
integrals
appearing
in the topological recursion formula.

When we consider a spectral curve
embedded in the principal $\bC^*$-bundle 
associated with $T^*C$, such as those 
we find in \cite{DFM, GS},
even though a similar
formalism works, the topological recursion
acquires a different mathematical flavor.
It is a relation to algebraic $K$-theory and
the Bloch regulator
appearing as a Bohr-Sommerfeld type quantization
condition
 described in \cite{ADKMV,GS}. We come back 
to this point later.

\subsection{Hitchin spectral curves}

In algebraic geometry, a spectral curve simply
means the diagram (\ref{eq:spectral}) for
an  algebraic curve $C$. The curve $\Sigma$ also
appress as the
\textbf{Seiberg-Witten curve} \cite{Gai}.
It is obvious that such a $\Sigma$ 
\emph{cannot} be the characteristic
variety of a $D$-module defined over 
the base curve $C$, because $\dim C=1$
and the characteristic varieties are necessarily
$\bC^*$ invariant with respect to the
$\bC^*$-action on $T^*C$.
To capture the geometry of a spectral curve,
we need to utilize Deligne's $\lam$-connections.
The idea of the $\lam$-connections is parallel to
that of  the WKB method in analysis. This is 
explained in Section~\ref{sect:lam}.

The notion of spectral curves was
 developed by Hitchin \cite{H1,H2} 
 in the
process of \textbf{Abelianization} of the 
moduli spaces of stable vector bundles on 
a  projective algebraic curve $C$
of genus greater than $1$
(see also \cite{BNR,DM,H,HT, LM2, S}).
Consider a \textbf{Higgs pair} $(E,\phi)$
consisting of a vector bundle $E$ of rank $r$ on 
$C$ and a \textbf{Higgs field} 
$\phi\in H^0(C,\cE nd(E)\tensor \Omega_C^1)$,
where $\Omega_C^1$ denotes the 
sheaf of holomorphic
$1$-forms on $C$. 
 The Higgs field here
is defined on a curve through the dimensional 
reduction of the Higgs boson 
\cite{Higgs} on a $4$-dimensional 
space-time. Let  $\eta$ be the tautological
$1$-form on $T^*C$ such that $-d\eta$ gives the
natural holomorphic symplectic form on $T^*C$.
Then the characteristic equation
$\det(\eta -\phi) = 0$
defines a spectral curve $\Sigma$ as the space
of eigenvalues of $\phi$. Under a good condition,
$\Sigma$ is nonsingular and the natural
projection $\pi:\Sigma\lrar C$ is a ramified
covering of degree $r$ with ramification 
divisor $R$. 
 In symplectic geometry, 
 a ramification point $p\in R$ is called 
 a \textbf{Lagrangian singularity}, and
 the branch divisor $\pi(R)\subset C$ the
 \textbf{caustics} of $\pi$.
Let $M\lrar \Sigma$ be the 
eigenspace bundle of the Higgs field on $\Sigma$,
and define $L=M\tensor \cO_\Sigma(R)$.
Then the original 
vector bundle $E$ is recovered by $E=\pi_* L$. 
The Abelianization refers to  the correspondence
$$
(C,E,\phi) \longleftrightarrow 
(\pi:\Sigma\lrar C, L,\iota^*\eta).
$$
Let us denote by
\begin{multline}
\label{eq:char coefficient}
s = (s_1,s_2,\dots,s_r) = 
\big(-\tr \phi, \tr(\wedge^2 \phi),
\dots,(-1)^r \tr(\wedge^r \phi)\big)
\\
\in V_{GL_r}^* :=
\bigoplus_{i=1}^r H^0\big(
C,(\Omega_C^1)^{\tensor i}\big)
\end{multline}
the characteristic coefficients of the Higgs field
$\phi$. 
The dual notation $^*$ on the vector space
is due to the analogy with the dual Lie algebra
we normally have as a target space of  a
moment map in real symplectic geometry.
In algebraic geometry, a family of groups can
act symplectomorphically, with the same Lie 
algebra. Here we have such a situation 
(see for example, \cite{HM}).
The notation $\tr(\wedge^i \phi)$ of a matrix
$\phi$ means the sum of all principal
$i\times i$-minors of $\phi$
that is considered as an element of 
the symmetric power
$H^0\big(
C,(\Omega_C^1)^{\tensor i}\big)$.
We are not talking about the exterior 
power $\phi\wedge \phi$ here, since all
higher exterior powers of $\phi$  
automatically vanish on $C$.
The global section
\begin{equation}
\label{eq:char equation}
\eta^{\tensor r} + \sum_{i=1} ^r
 \eta^{\tensor (r-i)}\tensor \pi^* s_i
\in 
H^0\big(T^*C\times V_{GL_r}^*,
\pi^*(\Omega_C^1)^{\tensor r}
\boxtimes \cO_{V_{GL_r}^*}\big)
\end{equation}
defines a family of spectral curves
\begin{equation}
\label{eq:family}
\begin{CD}
\Sigma_s\subset \widetilde{\Sigma}
 @>{\iota}>> T^*C\times V_{GL_r}^*@>>>
 V_{GL_r}^*
\\
&&@VV{\pi\times {\text{id}}}V
\\
C\times \{s\}@>>> C\times V_{GL_r}^*
\end{CD}
\end{equation}
on $V_{GL_r}^*$. The morphism
$\pi:\Sigma_s\lrar C$ has degree $r$.
When there is no need to specify the
rank $r$, we denote simply by
$V^*_{GL_r}=V^*_{GL}$.

Our discovery of this paper is that when 
we restrict ourselves to the case of 
$r=2$ and generic $s\in V_{GL_2}^*$
so that $\Sigma_s$ is smooth and the covering
is simply ramified,
the generalized
Eynard-Orantin theory precisely 
gives the quantization of a \emph{family} of
smooth spectral curves $\widetilde{\Sigma}\big|_V$
for a
contractible
 open neighborhood $V\subset V^*_{GL_2}$
of $s$.

\subsection{The Generalized Eynard-Orantin theory}

In their seminal paper \cite{EO1}, Eynard and
Orantin propose a geometric theory of
computing quantum invariants using an integral
recursion formula on a plane curve $\Sigma$
which is realized as
 a simply ramified covering $\pi:\Sigma\lrar \bC$,
 i.e., when the base curve
 $C$ of (\ref{eq:spectral})
 is the complex line $\bC$.
 In Section~\ref{sect:EO} we  generalize  the original
topological recursion to a mathematical 
 framework suitable for the purpose of the current paper.
 The heart of this theory is an integral recursion 
 formula, originally found in random matrix theory
 \cite{AMM, CEO, E2004}.

 The topological nature of the formula itself is 
 known to the mathematics community for a long
 time. It is the same degeneration on the 
 Deligne-Mumford moduli stack
 $\Mbar_{g,n}$ of $n$-pointed stable curves
 of genus $g$ as described in 
 \cite[Chapter 17, Section 5, Page 589]{ACG}.
 It appears as the Dijkgraaf-Verlinde-Verlinde
 formula \cite{DVV} for the Witten-Kontsevich 
 intersection theory \cite{K1992,W1991}, 
 known as the \emph{Virasoro constraint
 condition},
 and also as a recursion formula for the 
 Weil-Petersson volume of the moduli
 space of bordered hyperbolic surfaces
 in Mirzakhani's work \cite{Mir1,Mir2}
 (see also \cite{LX,MSaf}).
 The key difference between the topological recursion
 and the above mentioned  formalisms
 is that the former
  is a \textbf{B-model theory} that exhibits a
 universal structure
 (cf.~\cite{BKMP,M2}). 
 Indeed, the B-model formalism
 is the \emph{Laplace transform} 
 \cite{DMSS, EMS} of the 
 geometric equations mentioned above.
 
 In the context of the Hitchin spectral curves
or the Seiberg-Witten curves
 (\ref{eq:family}), the generalized 
 formalism goes 
 as follows. 
 The goal of the theory is to define 
 symmetric differentials $W_{g,n}^s$ on 
 $\Sigma_s^n$ for $g\ge 0$ and $n\ge 1$. The starting point
 is the two \emph{unstable} cases $2g-2+n\le 0$. 
 We first define $W_{0,1}^s(z_1) = \iota^* \eta(z_1)$,
 which is called the \textbf{Seiberg-Witten
 differential}. 
 As 
 $W_{0,2}^s(z_1,z_2)$ we take the Riemann
 fundamental form of the second kind
 \cite{Fay, Mumford} with an appropriate 
 normalization that
 we can choose on an open neighborhood 
 of a generic point $s\in V^*_{GL}$.
  This is the unique differential
 form of degree $2$ on $\Sigma_s\times \Sigma_s$ with 
 double poles along the diagonal, and when considered
 as an integration kernel it operates as the
 exterior differentiation $f\longmapsto df$ for
 any meromorphic function on $\Sigma_s$.
 For the \emph{stable} range
 $2g-2+n>0$, the differentials $W_{g,n}^s$
 at a point 
 $(z_1,\dots,z_n)\in \Sigma_s^n$
 are recursively
 defined by the following integral recursion
 formula:
 \begin{multline}
 \label{eq:EO intro}
 W_{g,n}^s(z_1,\dots,z_n) =
 \half\;\frac{1}{2\pi i}\sum_{p\in R_s}
 \oint_{\gam_p}
 \frac{\int_{z} ^{\sigma_p(z)}
 W_{0,2}^s(\;\cdot\; ,z_1)}
 {W_{0,1}^s\big(\sigma_p(z)\big)-W_{0,1}^s(z)}
 \\
 \times \Bigg[
 \sum_{j=2}^n \left(
 W_{0,2}^s(z,z_j)W_{g,n-1}^s\big(
 \sigma_p(z),z_{[\hat{1}, \hat{j}]}\big)
 +
 W_{0,2}^s\big(\sigma_p(z),z_j\big)
 W_{g,n-1}^s\big(
z,z_{[\hat{1}, \hat{j}]}\big)
 \right)
 \\
 +
 W_{g-1,n+1}^s\big(z,\sigma_p(z),z_{[\hat{1}]}
 \big)
 +
 \sum_{\substack{g_1+g_2=g\\
 I\sqcup J=[\hat{1}]}} ^{\text{stable}}
 W_{g_1,|I|+1}^s(z,z_I)W_{g_2,|J|+1}^s\big(
 \sigma_p(z),z_J\big)
 \Bigg].
 \end{multline}
 Here $R_s$ is the ramification divisor 
 of  the spectral curve $\pi:\Sigma_s\lrar C$
 which is assumed to be a simple 
 ramified covering,
 $\gam_p$ is a small simple closed loop
 with the positive orientation around
 a Lagrangian singularity 
 $p\in R_s\subset \Sigma_s$,
 and $\sigma_p$ is the local Galois 
 conjugation of the curve $\Sigma_s$ near $p$.
 The residue integration is taken with 
 respect to the $z$ variable on $\gam_p$.
 For the index set $[n]=\{1,\dots,n\}$, we indicate
 missing indices by the $\hat{}$ notation.
 For a subset $I\subset [n]$, 
 we denote $z_I = (z_i)_{i\in I}$, and by $|I|$
 the cardinality of $I$.
 The sum in the last line runs for all partitions
 of $g$ and set partitions of $\{2,\dots,n\}$,
 subject to the condition that 
 $2g_1-1+|I|>0$ and $2g_2-1+|J|>0$.

 The \textbf{free energy} of type $(g,n)$ is 
 a (meromorphic) function $F_{g,n}^s$ 
 on $\Sigma_s^n$
 satisfying that
 \begin{equation}
 \label{eq:Fgn intro}
 d_1\cdots d_n F_{g,n}^s = W_{g,n}^s.
 \end{equation}
 Of course such $F_{g,n}^s$'s are never unique
 because of the constants of integration, 
 and their existence is not even guaranteed because
 $\Sigma_s$  has a nontrivial fundamental group.
 When $F_{g,n}^s$ exists, we impose
 the uniqueness condition by
  integration along the fiber:
 \begin{equation}
 \label{eq:symmetry}
 (\pi_i)_* F_{g,n}^s :=
 \sum_{z_i\in \pi^{-1}(x_i)}
 F_{g,n}^s(z_1,\dots,z_i,\dots,z_n) = 0,
 \qquad (g,n) \ne (0,2).
 \end{equation}
 Here we choose an arbitrary point $x_i\in C$
 that is not a branched point, and consider the 
 integration of $F_{g,n}^s$  along the fiber
 of $\pi$ at $x_i$ 
 for the $i$-th component of the
 product of $\Sigma_s$, while fixing all other
 $z_j$'s, $j\ne i$.

\subsection{The main result}

We prove the following.

\begin{thm}[Main Theorem]
Let $C$ be an 
arbitrary  smooth projective algebraic
curve of genus $g \ge 2$ over $\bC$. We consider the 
family \eqref{eq:family} of
degree $2$ spectral curves on $C\times V_{SL_2}^*$
corresponding to the $SL(2,\bC)$ Hitchin
fibration.
If the spectral data $s\in V_{SL_2}^*
:=H^0\big(C,(\Omega_C^1)^{\tensor 2}
\big)$ is generic
so that $\Sigma_s$ is non-singular and the covering
$\pi:\Sigma_s\lrar C$ is simply ramified, then
there is an open neighborhood 
$s\in V\subset H^0\big(C,(\Omega_C^1)^{\tensor 2}
\big)$  such that the family of
spectral curves $\widetilde{\Sigma}\big|_V$
 is  quantizable by using the
 Eynard-Orantin theory.
   \end{thm}

More precisely, we construct a 
quantum curve, or 
a \textbf{Schr\"odinger 
operator} $P_s(x,\hbar)$, as more commonly known,
  on a formal family 
$C[[\hbar]]$ of the curve $C$ such that
\begin{equation}
\label{eq:D/DP intro}
E = \cD^\hbar\big/\cD^\hbar P_s
\end{equation}
is a $D$-module 
of $\cO_{C[[\hbar]]}$-rank $2$ over ${C[[\hbar]]}$.
Here we denote by 
$\cD^\hbar=\cD_{C[[\hbar]]}^\hbar$ the sheaf of 
differential operators on $C[[\hbar]]$
without $\hbar$-derivatives. We use
local coordinates $z$ on $\Sigma_s$,
$x$ on $C$, and a local section $z=z(x)$ of $\pi$.
 We prove that the canonical solution of
 the Schr\"odinger equation
\begin{equation}
\label{eq:Sch intro}
P_s(x,\hbar) \big|_U\Psi_s\big(z(x),\hbar\big) = 0
\end{equation}
defined on an open subset $U\subset C$
that contains no caustics of $\pi:\Sigma_s\lrar C$
is  constructed by the formula of \cite{EO1,GS}
\begin{equation}
\label{eq:Psi intro}
\Psi_s(z,\hbar) 
=
\exp
\left(
\sum_{g\ge 0} \sum_{n\ge 1}\frac{1}{n!}
\hbar^{2g-2+n} F_{g,n}^s(z,\dots,z)
\right).
\end{equation}
In the context of the AGT correspondence, 
this seems to be related to the function known as 
a \textbf{conformal block}.
We note that (\ref{eq:Psi intro})
 is exactly a geometric refinement
of the singular perturbation method
known as the WKB approximation.
Moreover, the \textbf{semi-classical limit}
(i.e., the zeroth-order terms in the $\hbar$-expansion
of the 
WKB approximation)
of this Schr\"odinger equation recovers
the spectral curve equation 
$$
\eta^{\tensor 2} 
+\pi^*s = 0
$$
for   $\Sigma_s\subset T^*C$.

The heart of the construction 
is Theorem~\ref{thm:Fgn recursion},
which is derived from the generalized integral
recursion 
(\ref{eq:EO intro}) by concretely evaluating the
residue integration of the formula. 
We emphasize that the residue calculation 
of (\ref{eq:EO intro}) is made
possible only because we generalize
the topological recursion formalism of \cite{EO1} to 
the compact base curve $C$. 
We establish the unique
existence of the free energy $F_{g,n}^s$ for every
$(g,n)\ne (0,2)$, and 
construct the Schr\"odinger operator
$P_s$ from (\ref{eq:Fgn recursion}), after
identifying $F_{0,2}^s$ through the first-order
WKB approximation.
Although in its expression, (\ref{eq:Psi intro}),
depends on the choice of coordinates, 
Theorem~\ref{thm:Fgn recursion} is coordinate
independent, and establishes the quantization 
of the spectral curve in a coordinate-free manner.

We also remark that though our formalism is 
 more general,  the actual technical
calculations are  parallel to that of
\cite{EB}. Indeed, we asked the following
question: what would be
the mathematical framework that would allow
the analysis 
technique of \cite{EB, DMSS,MS}  work?
In the process of answering this question, we
discover that the Hitchin spectral curves are
the right framework.

The $SL(2,\bC)$ assumption we impose is due to 
a technical reason, but not by any conceptual
reason. The formulation of 
\cite{BouE}, which  assumes that the spectral 
curve is a compact plane algebraic curve,
can easily be generalized
to our situation of Hitchin spectral curves
\eqref{eq:family}.
However, the idea developed in  \cite{BouE} 
does not seem to directly provide
the counterpart of  our
Theorem~\ref{thm:Fgn recursion}.
We can also allow a base curve $C$
with prescribed marked points, and
consider the moduli space of \emph{parabolic}
 Higgs bundles.
In the context of the AGT correspondence
and Seiberg-Witten curves
\cite{AGT,Gai}, such a setup naturally arises. 
In this 
paper, however, we stay with the simplest situation,
avoiding too much technical complications.
The case for parabolic Higgs bundles 
with singular Seiberg-Witten
differentials will be
treated in a forthcoming paper.

\subsection{The geometric significance
of the topological recursion}

The significance of what  the topological recursion
 does is first recognized  
in the string theory community
 \cite{BKMP,DV2007,M2,OSY}.
Mari\~no \cite{M2}, and then 
Bouchard, Klemm, Mari\~no, and Pasquetti 
\cite{BKMP}, have
conjectured that when the spectral curve 
$\pi:\Sigma\lrar \bC^*$ is 
the mirror curve of a toric Calabi-Yau space $X$
of dimension $3$ (in this case it covers the punctured 
complex line $\bC^*$),
the topological recursion should calculate open
Gromov-Witten invariants of $X$ for all genera
(\emph{the remodeling conjecture}).
Their conjecture is a concrete and
universal mechanism to read off, 
from  $W_{g,n}$ of
(\ref{eq:EO intro}), all   open Gromov-Witten
invariants of 
genus
$g$ with $n$ boundary 
components of the source Riemann surface
that are mapped to a Lagrangian in $X$.

Bouchard and Mari\~no then related the 
topological recursion with the counting problem of 
simple Hurwitz numbers \cite{BM}.
They conjectured that certain generating functions
of simple Hurwitz numbers should satisfy  
 (\ref{eq:EO intro}) for 
$C = \bC^*$ with the spectral curve $\Sigma$
 defined by the \emph{Lambert function}
$x = y e^{-y}$.

The Hurwitz number conjecture of Bouchard and 
Mari\~no was solved in \cite{EMS,MZ}. The key
discovery was that the topological recursion was 
equivalent to the \emph{Laplace transform}
of the combinatorial relation known as 
the \emph{cut-and-join equation} 
\cite{Goulden,GJ,V}
of Hurwitz numbers.
Here again, we emphasize that the proof of the
conjecture is based on the global complex 
analysis of the Lambert curve, rather than
the local behavior of the spectral curve.

Once the relation between a counting problem
(A-model) and the integral recursion on a 
complex curve (B-model) is understood as
the Laplace transform, the same idea is used to 
solve the remodeling conjecture
of \cite{BKMP} for the case of topological 
vertex \cite{Chen,Zhou3}. 
Since the topological vertex method gives a
combinatorial description of the Gromov-Witten 
invariants for an arbitrary
smooth toric Calabi-Yau threefold
\cite{LLLZ}, the smooth case of the 
remodeling conjecture
was solved in \cite{EO3} by identifying the 
combinatorial structure of the integral recursion
with the localization method in 
open Gromov-Witten
invariants. 
Most recently, the general orbifold
case of the  conjecture
is  solved in \cite{FLZ}.

The mathematical structure
 of topological recursion has also 
been  studied in 
 \cite{DBOSS, E2011}, when the 
 spectral curve is considered as a collection of
 disjoint open discs. In particular, the
 discovery of the equivalence to 
 the Givental formalism in this local case 
 \cite{DBOSS}, and its application to 
 obtaining a new proof of the ELSV formula
 \cite{DBKOSS},
 are significant. 
 Compared to these structural analysis,
 the emphasis of our current work lies in 
 noticing  the importance
 of  the
 \emph{global structure} of the spectral curve
 that covers an \emph{arbitrary} 
 projective algebraic curve.

 \subsection{Quantum curves, and the motivation of 
 our current paper}

Although the topological recursion  for 
simple Hurwitz numbers was conjectured
from the consideration of open 
Gromov-Witten invariants of $\bC^3$
at the infinity limit of the framing parameter,
the Hurwitz case has a feature not shared with the
geometry of toric Calabi-Yau spaces. This is 
\emph{the existence of the quantum curve}
\cite{MS}.  The similar situation  happens also
for orbifold Hurwitz numbers \cite{BHLM, MSS}.

Gukov and Su\l kowski \cite{GS} considered
the \textbf{A-polynomial} of Cooper, Culler, Gillet, 
Long, and Shalen \cite{CCGLS}
associated with a knot $K$. 
The $SL(2,\bC)$-character variety of the 
fundamental group of the knot 
complement is mapped to the boundary 
torus
\begin{multline*}
\Hom\big(\pi_1(S^3\setminus K),
SL(2,\bC)\big)\big/\!\!\big/SL(2,\bC)
\lrar
\Hom\big(\pi_1(T^2),
SL(2,\bC)\big)\big/\!\!\big/SL(2,\bC)
\\
\isom
(\bC^*)^2
\end{multline*}
and  determines a
(usually) 
singular plane algebraic curve in $(\bC^*)^2$
defined over $\bZ$. Its defining equation
is the A-polynomial, 
 which captures
the \emph{classical} knot invariant
$\pi_1(S^3\setminus K)$.
 The proposal of 
Gukov-Su\l kowski is that by applying the
topological recursion that is suitably
modified for spectral curves in
$(\bC^*)^2$, one can quantize the
A-polynomial into a Schr\"odinger equation,
much like (\ref{eq:Sch intro}) above but 
of an infinite order due to the appearance 
of $\bC^*$ in the fiber direction 
of $\pi$, whose
semi-classical limit recovers precisely the
A-polynomial. Moreover,
they predict that the Schr\"odinger equation
is equivalent to the \textbf{AJ-conjecture} of 
Garoufalidis \cite{Gar,GarLe}, which 
implies that the generator $\Psi$ of the
$\hbar$-deformed
$D$-module is the colored Jones polynomial
of the knot $K$!

We recall that the A-polynomial of a knot
$K$ is a polynomial
in $\bZ[x,y]$, where $x$ and $y$ are determined
by the meridian and the 
longitude of the torus boundary
of the knot complement in $S^3$. It is established
in \cite{CCGLS} that the Steinberg symbol 
$\{x,y\}\in K_2\big(\bC(C_K)\big)$ 
is a torsion element of the
second algebraic K-group of the function field
of the projective curve $C_K$ determined by the 
A-polynomial of the knot $K$. Gukov and
Su\l kowski \cite{GS}
attribute the quantizability of the 
A-polynomial to this algebraic K-theory condition,
which plays a similar role
of the Bohr-Sommerfeld quantization 
condition through the Bloch regulator.

We have constructed rigorous 
mathematical examples of the topological recursion
 in \cite{DMSS}, for which we can test
all physics predictions. A quantum curve construction
is also carried out in \cite{MS}, and for many
other examples of counting problems of Hurwitz type
\cite{BHLM,MSS,Zhou4}. For these cases the 
$K_2$ condition (the torsion property of the
Steinberg symbol) holds. But it has to be remarked
that all these rigorous examples have 
spectral curves of genus $0$. So far no
examples of quantum curves have been 
rigorously constructed for a spectral 
curve with a higher genus. 
This motivates our current paper. 
Although we do not address the question in this
paper, the ultimate interest is to identify the
quantum topological
 information that our $\Psi$ must carry.
In this context, establishing the relation to the 
Seiberg-Witten prepotential of 
Nekrasov \cite{N} through the AGT correspondence
\cite{AGT} is the key \cite{BMT,Gai}. 
The Eynard-Orantin theory then provides an 
expansion formula for the conformal block $\Psi$
from the geometric data of the
Seiberg-Witten curve covering
the Gaiotto curve.

We note that the relation between the topological
recursion
 and knot invariants are growing at this moment
\cite{AV, BE2, BEM,DFM, FGS1, FGS2}. 
It is beyond our scope to make any comment in 
this direction.

\subsection{Organization of the paper}
The paper is organized as follows. 
We begin with gathering the classical 
geometric materials we use in this paper, recalling
spectral curves, Riemann prime forms, and geometry
of degree 2 spectral curves, in
 Section~\ref{sect:spectral}.
Then in Section~\ref{sect:EO}, we re-define
the topological recursion with an arbitrary 
base curve. 
Section~\ref{sect:integral} is devoted to 
integrating the newly formulated
 recursion. We will establish
a differential recursion formula for free energies.
Here our generalization \eqref{eq:EO intro} of
 the topological 
recursion  of \cite{EO1} plays an essential role, 
due to the fact that our  spectral curve and the 
base curve are both compact.
The notion of quantum curves from 
physics requires us to utilize Deligne's
$\lam$-connections. We review the necessary
materials in Section~\ref{sect:lam}, following
\cite{A}.
Finally in Section~\ref{sect:WKB},
we take the \emph{principal specialization}
of the formula established in Section~\ref{sect:integral}. In this way we construct the quantum
curve and the $\hbar$-deformed
$D$-module, quantizing the
spectral curve. This method is indeed 
the same as solving the exact WKB analysis.

\section{Geometry of spectral curves}
\label{sect:spectral}

Let $C$ be a non-singular complete algebraic 
curve over $\bC$ of genus $g=g(C)\ge 2$. 
Although somewhat restrictive, since we need
the smoothness and the simple ramification
conditions, we adopt the 
following definition in this paper. 

\begin{Def}
A \textbf{spectral curve} of degree $r$ is a
 complete smooth algebraic curve 
$\Sigma$ embedded in 
the cotangent bundle $T^*C$ such that its projection
\begin{equation*}
\begin{CD}
\iota :\Sigma @>>> T^*C
\\
&&@VV{\pi}V
\\
&&C
\end{CD}
\end{equation*}
onto $C$ is a simply ramified covering of 
degree $r$. We denote by  
$\eta\in H^0\big(T^*C,\pi^*\Omega_C^1\big)$
the tautological $1$-form on $T^*C$ such that
$-d\eta$ is the canonical holomorphic
symplectic form on $T^*C$. A \textbf{spectral data}
is an element of a vector space
\begin{equation}
\label{eq:spectral data}
s=
(s_1,s_2,\dots,s_r)\in V_{GL}^* := \bigoplus_{i=1}^r 
H^0\big(C,(\Omega_C^1)^{\tensor i}\big)
\end{equation}
of dimension $r^2(g-1)+1$. We consider a
spectral data \textbf{generic} if the 
characteristic equation
\begin{equation}
\label{eq:char eq}
\eta^{\tensor r}+
\sum_{i=1}^r s_i \eta^{\tensor (r-i)} = 0
\end{equation}
defines a spectral curve $\Sigma$
in our sense. Here  the characteristic polynomial 
is viewed as a global section
$$
\eta^{\tensor r}+ 
\sum_{i=1}^r \pi^*s_i \tensor \eta^{\tensor (r-i)}\in 
H^0\big(T^*C,\pi^*(\Omega_C^1)^{\tensor r}\big)
$$
that defines $\Sigma$ as its $0$-locus.
To indicate the $s\in V^*_{GL}$ dependence of
the spectral curve, we use the notation 
$\Sigma = \Sigma_s$.
\end{Def}

\begin{rem}
The smoothness assumption of $\Sigma_s$ is 
crucial. The evaluation of the
residue integrations of (\ref{eq:EO intro}) 
that is necessary for defining the free energies
would not go through if $\Sigma_s$ has singularities.
The assumption of simple ramification is 
imposed here only because of the simplicity of the 
formulation. We can generalize the framework
to arbitrarily  ramified coverings in a similar way
as developed in \cite{BouE}, although it is 
restricted to the case when 
the spectral curve is a compact plane curve.
\end{rem}

\begin{rem}
Note that for 
every $1$-form $s_1\in H^0(C,\Omega_C^1)$,
$\eta+\pi^*s_1$ determines the same
symplectic form, because
\begin{equation}
\label{eq:eta+s1}
-d\eta = -d(\eta + \pi^* s_1).
\end{equation}
\end{rem}

The spectral curves are  
originally considered in the 
context of Abelianization of the moduli space of
stable vector bundles on $C$ in terms of 
Hitchin integrable systems \cite{BNR, DM, 
H1,H2}. Recall that a Higgs pair $(E,\phi)$ 
of rank $r$ and degree $d$ consists
of a vector bundle $E$ on $C$ of 
rank $r$ and degree $d$ and a Higgs field
$\phi\in H^0\big(C,\cE nd (E)\tensor \Omega_C^1
\big)$. Stability conditions are appropriately 
defined so that for the case of $(r,d)=1$ the
moduli space $\cH_C(r,d)$ 
of stable Higgs pairs form  a smooth
quasi-projective variety of dimension
$2(r^2(g-1)+1)$. 
The space  $\cH_C(r,d)$ contains the cotangent
bundle $T^*\cU_C(r,d)$ of the moduli space
$\cU_C(r,d)$ of stable vector bundles of 
rank $r$ and degree $d$ on $C$ as an open dense
subset.
We note that the character variety
$$
\Hom\big(\pi_1(C),GL(r,\bC)\big)\big/\!\!
\big/GL(r,\bC)
$$
has the same dimension $2(r^2(g-1)+1)$.
We refer to \cite{H,HT, HM} for more detail 
on the relation between the character variety
and the Hitchin moduli spaces.

The \textbf{Hitchin fibration}
\begin{equation}
\label{eq:Hitchin map}
\mu_H:\cH(r,d)\owns (E,\phi) \longmapsto
\det(y-\phi) = y^r + \sum_{i=1}^r (-1)^i
\trace(\wedge^i \phi) y^{r-1}
\in V_{GL}^*
\end{equation}
induces an algebraically completely integrable
 Hamiltonian system on $\cH_C(r,d)$.
 A generic Higgs pair $(E,\phi)$ gives rise to a
 generic spectral data
$$
s=(s_1,s_2,\dots,s_r) =  \big((-1)^i
\trace(\wedge^i \phi) \big)_{i=1} ^r \in V_{GL}^*,
$$
 and the fiber of the Hitchin fibration $\mu_H$
 is isomorphic to the Jacobian variety of the
 spectral curve:
 $$
 \mu_H^{-1}(s) \isom \Jac(\Sigma_s).
 $$
 In particular, the spectral curve has genus
 \begin{equation}
 \label{eq:g(Sigma)}
 \hat{g} = g(\Sigma_s) = r^2(g-1)+1.
 \end{equation}
 If we further assume that the projection
 $\pi:\Sigma_s\lrar C$ is simply ramified, 
 then the ramification divisor $R_s\subset \Sigma_s$
 consists of $2r(r-1)(g-1)$ points. This shows 
 that the spectral curves we are dealing with
 form a very 
 special class of ramified coverings over $C$
 of a given degree $r$. If we were to 
 consider the Givental formalism following 
 \cite{DBOSS} or the corresponding Frobenius
 manifold \cite{D, DZ}, then for a fixed $C$, 
 the 
 cardinality of $R_s$ should
 represent  the degrees of freedom
 of the theory. However, we note that 
 $R_s$ is far from arbitrary as
 a divisor. Indeed the degrees
 of freedom of our case is less than the expected
 value from the Frobenius manifold theory, since
 $$
 \dim V^*_{GL} - \dim\Jac(C) 
 = (r^2-1)(g-1) < (2r^2-2r)(g-1) = \deg R_s
 $$
 for $r\ge 2$. Here we subtract the dimension 
 of $\Jac(C)$ because changing the vector bundle
 $E$ to $E\tensor \cL$ with $\cL\in \Jac(C)$ does 
 not change the spectral curve, because 
 the Higgs field $\phi$ remains the same. 
 As noted in \cite{HM}, the family 
 of spectral curves is effective only on the space
 \begin{equation}
 \label{eq:VSL}
 V^*_{SL} := \bigoplus_{i=2}^r H^0
 \big(C,(\Omega_C^1)^i\big),
 \end{equation}
 which has the dimension $(r^2-1)(g-1)$.

 This consideration also corresponds to the 
 following.
 The application of a 
 \emph{symplectic transformation}
 $\eta\longmapsto \eta+\frac{1}{r} \pi^*s_1$ 
  changes
 the characteristic equation
 \begin{equation}
 \label{eq:eta shift}
 \eta^{\tensor r} +\sum_{i=1}^r
 \pi^* s_i \tensor \eta^{\tensor(r-i)}
= \left(\eta + \frac{1}{r} \pi^*s_1\right)^{\tensor r} +
\sum_{i=2}^r
\pi^*s_i'\tensor  
 \left(\eta + \frac{1}{r} \pi^*s_1\right)^{\tensor(r-i)},
 \end{equation}
 where $s_i'\in 
 H^0\big(C,(\Omega_C^1)^{\tensor i}\big)$ is a 
 polynomial in $s_1,\dots,s_i$ of 
 the homogeneous 
 degree $i$. Thus without loss of generality 
 we can consider the \textbf{traceless}
 spectral data
 $
s= (s_2,\dots,s_r)\in V_{SL}^*
 $
 for the purpose of dealing with the spectral 
 curve.

 To introduce the Eynard-Orantin theory, we need
 a classical geometric ingredient,
  the \textbf{normalized
fundamental differential
of the second kind} $B_X(z_1,z_2)$ on a 
smooth complete algebraic curve $X$ 
\cite[Page 20]{Fay}, \cite[Page 3.213]{Mumford}. 
This is a symmetric
differential $2$-form on $X\times X$
with second-order poles only along the diagonal. 
We identify the Jacobian
variety of $X$ as 
$\Jac(X) = Pic^0(X)$, which is isomorphic to
$Pic^{g-1}(X)$. The \emph{theta divisor}
 $\Theta$ of 
$Pic^{g-1}(X)$ is defined by 
$$
\Theta = \{L\in Pic^{g-1}(X)\;|\; \dim H^1(X,L)>0\}.
$$
We use the same notation for the translate divisor
on $\Jac(X)$, also called the theta divisor. 
Consider the diagram
\begin{equation*}
		\xymatrix{& \Jac(X)
		\\
		&X\times X\ar[dl]_{pr_1} \ar[u]_{\delta}
		\ar[dr]^{pr_2}
		\\
		X& & 	X ,
		}
\end{equation*} 
where $pr_j$ denotes the projection to the
$j$-th component,
and 
$$
\delta:X\times X \owns (p,q)\longmapsto p-q\in
\Jac(X).
$$
The \emph{prime form} $E_X(z_1,z_2)$ 
\cite[Page 16]{Fay} is
defined as a holomorphic section 
$$
E_X(p,q) \in H^0\left(X\times X,
pr_1^* (\Omega_X^1)^{-\half}\tensor 
pr_2^* (\Omega_X^1)^{-\half}\tensor
\delta^*(\Theta)
\right),
$$
where we choose  Riemann's spin structure
(or the Szeg\"o kernel) $(\Omega_X^1)^\half$, 
which has a unique global section
up to the constant multiplication
(see \cite[Theorem~1.1]{Fay}).
We have
\begin{enumerate}
\item $E_X(p,q)$ vanishes only along the
diagonal $\Delta\subset X\times X$, and has
simple zeros along $\Delta$.
\item Let $z$ be a local coordinate on 
$X$. Then $dz(p)$ gives the local trivialization of
$\Omega_X^1$ around $p$. When $q$ is near at $p$,
 $\delta^*(\Theta)$ is also trivialized 
around $(p,q)\in X\times X$, and we have a local
expression
\begin{equation}
\label{eq:E}
E_X\big(z(p),z(q)\big) = 
\frac{z(p)-z(q)}{\sqrt{dz(p)}\cdot
\sqrt{dz(q)}}\left(1+O\big((z(p)-z(q))^2\big)
\right).
\end{equation}
\item $E_X\big(z(p),z(q)\big) 
= -E_X\big(z(q),z(p)\big) $.
\end{enumerate}
The fundamental $2$-form $B_X(p,q)$
is then defined by
\begin{equation}
\label{eq:B}
B_X(p,q) = d_1\tensor d_2 \log E_X(p,q)
\end{equation}
(see \cite[Page 20]{Fay}, 
\cite[Page 3.213]{Mumford}).
We note that $dz(p)$ appears in (\ref{eq:E}) just
as the indicator of our choice of the
local trivialization. 
With this local trivialization, we have
\begin{multline}
\label{eq:B local}
B_X\big(z(p),z(q)\big) =
d_1\tensor d_2 \log E\big(z(p),z(q)\big)
\\
=
\frac{dz(p)\cdot dz(q)}{\big(z(p)-z(q)\big)^2}
+O(1)\;dz(p)\cdot dz(q)
\\
\in H^0\left(X\times X,pr_1^*\Omega_X^1\tensor
pr_2^* \Omega_X^1\tensor \cO(2\Delta)
\right).
\end{multline}

As noted in the literature \cite{Fay, Mumford}, 
the local expression (\ref{eq:B local}) alone
does not uniquely determine 
the form. Riemann
chose a symplectic basis 
$\la A_1,\dots,A_g;B_1,\dots,B_g\ra$ 
for $H_1(X,\bZ)$, and normalized the
fundamental form by
\begin{equation}
\label{eq:B-normal}
\oint_{A_j} B_X(\;\cdot\;,q) = 0
\end{equation}
for every $A$-cycle $A_j$, $j=1, \dots,g$. 
 Because of the symmetry $B_X(p,q)=B_X(q,p)$,
 the $A$-cycle normalization uniquely determines
 the fundamental form.

 In the theory of complex analysis in one variable,
 the most fundamental object is the 
 Cauchy integration kernel. Ironically,
 we do not have a Cauchy kernel on a compact
 Riemann surface $X$. The best we can do is the
 meromorphic $1$-form $\omega^{a-b}(z)$ 
 uniquely defined
 by the following conditions. Let $a$ and $b$ be
 two distinct points of $X$. 
 \begin{enumerate}
 \item $\omega^{a-b}(z)$ is holomorphic
 except for $z=a$ and $z=b$.
 \item $\omega^{a-b}(z)$ has a simple pole of residue 
 $1$ at $z=a$.
 \item $\omega^{a-b}(z)$ has a simple pole of
 residue $-1$ at $z=b$. 
 \item $\omega^{a-b}(z)$ is $A$-cycle normalized:
 $$
 \oint_{A_j} \omega^{a-b}(z) = 0
 $$
 for every $j=1,\dots,g$.
 \end{enumerate}
 The relation between $\omega^{a-b}(z)$ and
 Riemann's normalized second fundamental 
 form is
 \begin{equation}
 \label{eq:B and omega}
 d_1 \omega^{z_1-b} (z_2) = B_X(z_1,z_2).
 \end{equation}
 This equation does not depend on 
 the point $b\in X$.

 Now let us go back to our spectral curve
 \begin{equation}
 \label{eq:Cs}
\begin{CD}
\iota :\Sigma_s @>>> T^*C
\\
&&@VV{\pi}V
\\
&&C
\end{CD}.
\end{equation}
 In what follows, we concentrate our attention to the
 case of $r=2$ traceless spectral data. Thus our
 spectral curve $\Sigma=\Sigma_s$ is a double sheeted ramified
 covering of $C$ defined by a characteristic equation
 \begin{equation}
 \label{eq:r=2}
 \eta^{\tensor 2} + \pi^*s_2 = 0,
 \end{equation}
 where the spectral data $s$ consists of only one
 component
 $s = s_2\in H^0\big(
 C,(\Omega_C^1)^{\tensor 2}\big)$,
 which is a generic
 quadratic differential on $C$ so that the characteristic
 equation defines a smooth curve that is simply 
 ramified over $C$. 
The genus of the spectral curve,
 calculated by (\ref{eq:g(Sigma)}), gives
$ \hat{g} = g(\Sigma_s) = 4g-3$.
 The cotangent bundle $T^*C$ has a natural
 involution
 \begin{equation}
 \label{eq:sigma}
 \sigma:T^*C\supset T_x^*C\owns (x,y)\longmapsto
 (x,-y)\in T_x^*C\subset T^*C.
 \end{equation}
 The spectral curve $\Sigma_s$ is invariant under 
 $\sigma$, and it provides the deck-transformation
 of the ramified covering $\pi:\Sigma_s\lrar C$.

 Let $R_s\subset \Sigma_s$ denote the ramification divisor
 of this covering. Because of the simple
 covering assumption, $R_s$ as a point set has
 $4g-4$ distinct points that are determined by
 $s_2=0$ on $C$. Since both $C$ and $\Sigma_s$ 
 are divisors of $T^*C$, $R_s$ is defined
 also as $C\cap \Sigma_s$. Note that $\eta$ vanishes
 only along $C\subset T^*C$. As a holomorphic
 $1$-form on $\Sigma_s$, $\iota^*\eta$ has 
 $2\hat{g}-2 = 8g-8$ zeros on $\Sigma_s$. Thus it
 has a degree $2$ zero at each point of $R_s$.

 As mentioned above, the Eynard-Orantin
 theory requires a \emph{normalized}
 second fundamental form of Riemann. 
 To normalize differential forms, there are
 many different choices. Here we use the 
 $A$-cycle normalization, following 
 Riemann's original idea. The reason for this 
 choice is its extendability to a
\emph{family of smooth}
 spectral curves 
 $$\widetilde{\Sigma}\big|_V
 =\{\Sigma_s\}_{s\in V}
 $$
 on a contractible open subset 
 $V\subset H^0\big(
 C,(\Omega_C^1)^{\tensor 2}\big)$.

  To explain our choice of the symplectic basis
  of the first homology group of the family
  of spectral curves, let us start with 
  choosing, once and for all, a symplectic
  basis 
  $$
  \la A_1,\dots,A_g;B_1,\dots,B_g\ra  =
  H_1(C,\bZ).
  $$
 Let us label points of $R_s$ and denote
 $R_s=\{p_1,p_2,\dots,p_{4g-4}\}$. We can 
 connect $p_{2i}$ and $p_{2i+1}$, $i=1, \dots,
 2g-3$, with a simple
 path on $\Sigma_s$ 
 that is mutually non-intersecting so that
 $\pi^*(\overline{p_{2i}p_{2i+1}})$,
 $i=1, \dots,
 2g-3$, 
 form a part of the basis for $H_1(\Sigma_s,\bZ)$. 
 We denote these cycles by $\a_1,\dots,\a_{2g-3}$.
 Since $\pi$ is locally homeomorphic away from 
 $R_s$, we have $g$ cycles $a_1,\dots,a_g$ on
 $\Sigma_s$ so that $\pi_*(a_j) = A_j$ for
 $j=1,\dots,g$, where $A_j$'s are previously
 chosen $A$-cycles of $C$. 
 We define the $A$-cycles of $\Sigma_s$ to be the
 set 
 \begin{equation}
 \label{eq:A on Cs}
 \{a_1,\dots,a_g,\sigma_*(a_1),\dots,\sigma_*(a_g),
 \a_1,\dots,\a_{2g-3}\}\subset H_1(\Sigma_s,\bZ).
\end{equation}
Clearly, this set can be extended into a symplectic
basis for $H_1(\Sigma_s,\bZ)$. 
This choice of the symplectic basis 
trivializes the homology bundle
$$
\big\{H_1(\Sigma_s,\bZ)\big\}_{s\in V}
\lrar V\subset H^0\big(
 C,(\Omega_C^1)^{\tensor 2}\big)
 $$
 globally on a contractible $V$. 
 
 The monodromy of the choice of the symplectic
 basis on the family of \emph{all}
 smooth spectral curves leads us to considering
 the modular group action on the space of 
 solutions to the Eynard-Orantin theory 
 (\ref{eq:EO intro}). In this paper we stay
 with the family on a contractible base.

 \section{The Eynard-Orantin integral recursion
on an arbitrary base curve}
\label{sect:EO}

The construction of the 
$\hbar$-deformed $D$-module over an
arbitrary complete smooth curve $C$ is carried
out in three stages. 
\begin{enumerate}
\item Construction of the Eynard-Orantin differentials
$W_{g,n}^s$ on  $\Sigma_s ^n$ for
all $g\ge 0$ and $n \ge 1$ using the geometry of
the spectral curve $\Sigma_s$.
\item Construction of the free energies
$F_{g,n}^s$, which are meromorphic functions
on $\Sigma_s ^n$ for $2g-2+n> 0$, and satisfies that
$d_1\cdots d_n F_{g,n}^s = W_{g,n}^s$.
\item Construction of the exponential 
generating function of $F_{g,n}^s$
with $\hbar$ as the expansion 
parameter, in the way the WKB approximation
dictates us to do, and take its
principal specialization. 
The principal specialization then gives the
generator of the $\hbar$-deformed $D$-module. 
\end{enumerate}

Our point of departure is the spectral curve
(\ref{eq:family}) defined by the characteristic
equation (\ref{eq:char equation}) for generic values of a
spectral data $s=V^*_{GL}$
so that $\Sigma_s$ is smooth and the covering
$\pi$ is simply ramified along the divisor 
$R_s$. Since we do not consider the monodromy
transformation and the modular property
of the theory under the change of symplectic 
basis for $H_1(C,\bZ)$ in the current paper,
for simplicity we assume that $s$ belongs to 
a contractible open subset  
$V\subset V^*_{SL} = 
\bigoplus_{i=2}^r H^0
\big(C,(\Omega_C^1)^{\tensor i}\big)$
(\ref{eq:VSL}).
What we call the Eynard-Orantin theory in 
this paper is the following procedure of
determining the Eynard-Orantin differentials.

\begin{Def}[Eynard-Orantin differentials]
For every $(g,n)$, $g\ge 0$ and $n\ge 1$, 
the quantity $W_{g,n}$ defined by one of the 
following formulas is what we call the
\textbf{Eynad-Orantin differential} of
type $(g,n)$. 
To avoid extra cumbersome notation, we
suppress the $s$-dependence of 
the Eynard-Orantin differentials.
First, we define a holomorphic $1$-form on the
spectral curve $\Sigma_s$ by 
\begin{equation}
\label{eq:W01}
W_{0,1}(z_1) = \iota^* \eta \in H^0(\Sigma_s,\pi^*
\Omega_C^1)
\subset H^0(\Sigma_s,
\Omega_{\Sigma_s}^1).
\end{equation}
We define a symmetric $2$-form $W_{0,2}$
on $\Sigma_s\times \Sigma_s$ using 
Riemann's normalized second fundamental form by
\begin{equation}
\label{eq:W02}
W_{0,2}(z_1,z_2) = B_{\Sigma_s}(z_1,z_2),
\end{equation}
where $(z_1,z_2)\in \Sigma_s\times \Sigma_s$.
For this definition we choose
once and for all a symplectic basis for 
$H_1(\Sigma_s,\bZ)$ that is independent
of $s\in V$ and  use 
the $A$-cycle normalized  second fundamental 
forms of Section~\ref{sect:spectral}.

For each $p\in R_s$ we 
choose a local neighborhood $p\in U_p\subset
\Sigma_s$. Since the covering is simple, 
there is a local Galois conjugation 
\begin{equation}
\label{eq:Galois}
\sigma_p:U_p\lrar U_p,
\end{equation}
which is an involution.
We define the 
\textbf{recursion kernel} for each 
$p\in R_s$ by
\begin{multline}
\label{eq:kernel}
K_p(z,z_1) = \frac{\int_{z}^{\sigma_p(z)}
B_{\Sigma_s}(\;\cdot\;,z_1)}
{\sigma_p^*W_{0,1}(z)-W_{0,1}(z)}
\\
\in H^0\left(
U_p\times \Sigma_s, \big((\Omega_{\Sigma_s}^1)^{-1}(2R_s)
\boxtimes \Omega_{\Sigma_s}^1\big)
\tensor \cO_{U_p\times \Sigma_s}(\Delta_s + 
(\sigma_p\times id)^*
\Delta_s)
\right),
\end{multline}
where  $\Delta_s\subset \Sigma_s\times \Sigma_s$ is the diagonal.
The reciprocal notation means
$$
\frac{1}{W_{0,1}(z)} \in H^0\big(
\Sigma_s,(\Omega_{\Sigma_s}^1)^{-1}
\tensor \cO_{\Sigma_s}(2R_s)\big).
$$
Using the recursion kernel, we define the 
first two Eynard-Orantin differentials in
the stable range $2g-2+n>0$.
\begin{equation}
\label{eq:W11}
W_{1,1}(z_1) = \half\;
\frac{1}{2\pi i}\sum_{p\in R_s}
\oint_{\gam_p} K_p(z,z_1) B_{\Sigma_s}(z,\sigma_p(z)),
\end{equation}
\begin{multline}
\label{eq:W03}
W_{0,3}(z_1, z_2,z_3) = \half\;
\frac{1}{2\pi i}\sum_{p\in R_s}
\oint_{\gam_p} K_p(z,z_1) 
\\
\times
\big(
B_{\Sigma_s}(z,z_2)B_{\Sigma_s}(\sigma_p(z),z_3)
+
B_{\Sigma_s}(z,z_3)B_{\Sigma_s}(\sigma_p(z),z_2)
\big).
\end{multline}
Here and in what follows, $\gam_q$ denotes a 
positively oriented simple
closed loop around a point $q\in \Sigma_s$,
and the integration is taken with respect to the
variable $z$ along the loop
$\gam_p$ for each $p\in R_s$. For a general value
of $(g,n)$ subject to $2g-2+n\ge 2$, the 
Eynard-Orantin differential is recursively 
defined by
\begin{multline}
\label{eq:Wgn}
W_{g,n}(z_1,\dots,z_n)
=
\half\;
\frac{1}{2\pi i}\sum_{p\in R_s}
\oint_{\gam_p} K_p(z,z_1)
\\
\times
\Bigg[
\sum_{j=2}^n \bigg(
W_{0,2}(z,z_j) 
W_{g,n-1}(\sigma_p(z),z_{[\hat{1},\hat{j}]})
+
W_{0,2}(\sigma_p(z),z_j) 
W_{g,n-1}(z,z_{[\hat{1},\hat{j}]})
\bigg)
\\
+
W_{g-1,n+1}\big(z,\sigma_p(z),z_{[\hat{1}]}\big)
+
\sum_{\substack{g_1+g_2=g
\\
I\sqcup J = \{2,\dots,n\}}}
^{\text{stable}}
W_{g_1,|I|+1}(z,z_I)W_{g_2,|J|+1}(\sigma_p(z),z_J)
\Bigg].
\end{multline}
Here we use the index convention that 
$[n] = \{1,\dots,n\}$, the hat notation $[\hat{j}]$ 
indicates deletion of the index, and for every subset
$I\subset [n]$, $z_I = (z_i)_{i\in I}$, and 
$|I|$ is the cardinality of the subset.
The sum in the third line is for indices in the 
stable range only.
\end{Def}

\begin{rem}
$W_{0,1}$ is also known as the Seiberg-Witten 
differential, when we allow prescribed poles
of $s$ on $C$. In this paper we consider only
holomorphic  $s$. Spectral data
with poles will be dealt with in a forthcoming
paper.
\end{rem}

In this definition, we need to clarify 
the ambiguity of the
integration in (\ref{eq:kernel}). 
Since $\Sigma_s$ has genus $r^2(g-1)+1$, the integration
from $z$ to $\sigma_p(z)$ of any $1$-form is
ambiguous. We use a systematic method to avoid
this ambiguity. 
Let us recall the unique $A$-cycle normalized 
meromorphic $1$-form
$\omega_s^{z-b}(z_1)$ on $\Sigma_s$. 
Regardless the point $b\in \Sigma_s$, we have
$d_z \omega_s^{z-b}(z_1) = B_{\Sigma_s}(z,z_1)$.
Therefore, we \emph{define}
the integral to be
\begin{equation}
\label{eq:B integral}
\int_{z}^{\sigma_p(z)} B_{\Sigma_s}(\;\cdot\;,z_1)
= \omega_s^{\sigma_p(z)-b}(z_1) -
\omega_s^{z-b}(z_1) 
= \omega_s^{\sigma_p(z)-z}(z_1).
\end{equation}
The recursion kernel
is now calculated to be
\begin{equation}
\label{eq:kernel1}
K_p(z,z_1) = \frac{\omega_s^{\sigma_p(z)-z}(z_1)}
{\sigma_p^*\eta(z)-\eta(z)}.
\end{equation}
From now on we omit the pull-back sign $\iota^*$
by the
inclusion $\iota :\Sigma_s\lrar T^*C$.

\begin{rem}
The existence of a canonical  choice of the 
integral (\ref{eq:B integral})
for the family of spectral curves
$\widetilde{\Sigma}\big|_V=\{\Sigma_s\}_{s\in V}$
 is significant
for the existence of the quantum curve,
starting from the recursion formula (\ref{eq:Wgn}).
Our choice of the trivialization of the homology
bundle $\{H_1(\Sigma_s,\bZ)\}_{s\in V}$
that we have made in the end of 
Section~\ref{sect:spectral} assures this unique 
existence. 
\end{rem}

\begin{rem}
Recently many calculations have been performed
to relate the Eynard-Orantin differentials
with intersection numbers of 
certain tautological classes on $\Mbar_{g,n}$
\cite{DBOSS, E2011}. All these calculations 
assume that the spectral curve is a ramified 
covering over $\bC$, and that the curve itself is 
just the disjoint union of small disks around
each ramification point. The location of 
these ramification points are arbitrarily chosen to 
represent the degree of freedom for deformations.

Here we emphasize that the spectral
curve $\Sigma_s$ is a \emph{global} object, and
that the ramification divisor $R_s$ on $\Sigma_s$ is not
an arbitrary set of  points. We view that
the heart of the Eynard-Orantin theory lies in the
global structure of the spectral curve, and hence
the calculation of the residues appearing in the
definition above has to be carried out globally,
not locally. In what follows, we perform this
very calculation. 

The relation between 
the local and global considerations
mentioned above gives us a non-trivial formula
of the result of our calculations in terms of 
tautological intersection numbers on $\Mbar_{g,n}$.
The identification of this formula is one of the
important questions that is not addressed in the
current paper.
\end{rem}

To actually compute
integrals, it is convenient 
to consider the case when both $z$ and $z_1$ are
close to a ramification point $p\in R_s$, but not 
quite equal. Then we have local expressions
\begin{align}
\label{eq:local omega}
\omega_s^{\sigma_p(z)-z}(z_1) 
&= \left(
\frac{1}{z_1-\sigma_p(z)}-\frac{1}{z_1-z} +O(1)
\right)dz_1 ,
\\
\label{eq:local B}
B_{\Sigma_s}(z,\sigma_p(z)) 
&=
\left(
 \frac{1}{(z-\sigma_p(z))^2}+O(1)
 \right)dzd\sigma_p(z),
 \\
 \label{eq:local eta}
 \eta(z) 
 &= h(z)dz.
\end{align}
We can also choose a small neighborhood of
$p$ such that 
\begin{equation}
\label{eq:local sigma}
\sigma_p(z) = -z,
\end{equation}
if necessary. In this case $z=0$ is the 
point $p\in R_s$. We also use formulas
\begin{equation}
\begin{aligned}
\label{eq:BK involution}
K_p(z,z_1) &= K_p(\sigma_p(z),z_1) 
= -K_p(z,\sigma_p(z_1)),
\\
B_{\Sigma_s}(z_1,\sigma_p(z_2)) &=
B_{\Sigma_s}(\sigma_p(z_1),z_2),
\\
h(\sigma_p(z)) &= h(z).
\end{aligned}
\end{equation}

\begin{prop}
\label{prop:Wgn properties}
For $2g-2+n>0$, the Eynard-Orantin
differential $W_{g,n}(z_1,\dots,z_n)$ is a 
symmetric meromorphic $n$-form on 
$\Sigma_s ^n$ with poles only at $z_i\in R_s$, 
$i=1, \dots, n$. It satisfies the following
balanced average property with respect to the
deck transformation:
\begin{equation}
\label{eq:Wgn balance}
\sum_{z_i\in \pi^{-1}(x_i)}
W_{g,n}(z_1,\dots,z_i,\dots,z_n)
= 0,\quad i=1, \dots, n,
\quad (g,n)\ne (0,2).
\end{equation}
Here we choose a non-branched point
$x_i\in C$, and  add 
 $W_{g,n}(z_1,\dots,z_i,\dots,z_n)$
for all $r$-points $z_i\in \pi^{-1}(x_i)$
on the fiber of $x_i$. (This is commonly known 
as the integration along the fiber.) 
\end{prop}

\begin{proof}
Since the assertion of the Proposition is
essentially a local statement, we can
take an affine covering of
 the base curve $C$,
 and prove the statement on each
affine piece. 
Although the proof is quite involved
and requires many steps
for an affine curve, the idea and
the technique are exactly the same
as those  in \cite{EO1}.
\end{proof}

\section{The differential recursion for
free energies}
\label{sect:integral}

The global property of the spectral curve we
are emphasizing in this paper is that we
can actually \emph{integrate and evaluate} the residue
calculations appearing in the definition 
of the Eynard-Orantin differentials. The purpose
of this section is to concretely perform this
evaluation. We start with giving the definition of
free energies. It is worth mentioning that
all our calculations are actually performed on 
the  family of spectral curves defined
on a contractible base space $V$ as
explained in Section~\ref{sect:spectral}.
Again to avoid cumbersome notations, we
suppress the $s$-dependence in what follows.

\begin{Def} 
The \textbf{free energy} of type $(g,n)$ is a 
function $F_{g,n}(z_1,\dots,z_n)$
defined on $\Sigma_s^n$ subject to the following two
conditions:
\begin{align}
\label{eq:primitive}
d_1\cdots d_n F_{g,n}(z_1,\dots,z_n)
&= W_{g,n}(z_1,\dots,z_n),
\\
\label{eq:Fgn balance}
\sum_{z_i\in \pi^{-1}(x_i)}
F_{g,n}(z_1,\dots,z_i,\dots,z_n)
&= 0,\quad i=1, \dots, n,
\quad (g,n)\ne (0,2).
\end{align}
Here we choose a non-branched point
$x_i\in C$, and  consider the integration 
of $F_{g,n}$ along the fiber of $x_i$ 
with respect to 
the projection $\pi:\Sigma\lrar C$ applied to the 
$i$-th component. 
\end{Def}

\begin{rem}
The primitive condition (\ref{eq:primitive}) alone
does not determine $F_{g,n}$ due to constants
of integration. For example, one can add any
function in less than $n$ variables to $F_{g,n}$. 
It is obvious that the vanishing condition of the
integration along the fiber
(\ref{eq:Fgn balance}), reflecting 
(\ref{eq:Wgn balance}),
uniquely determines the free energies.
The authors are indebted to Paul Norbury 
and Brad Safnuk for the idea
of imposing (\ref{eq:Fgn balance}) to define the 
unique free energies.
In the examples considered in
\cite{DMSS}, we know $F_{g,n}$  
from the beginning because
we start with an A-model counting problem that
defines the free energies via the Laplace transform.
In our current context, since we start with the
Eynard-Orantin theory, i.e., from the B-model side, 
we have no knowledge of what the corresponding
A-model is. 
\end{rem}

\begin{rem}
We exclude the case $(g,n) = (0,2)$ from the
balanced Galois
 average condition (\ref{eq:Fgn balance}). 
How to define $F_{0,2}$ is an extremely subtle
matter, and is also related to the heart of the
quantizability of the spectral curve $\Sigma_s$. 
We discuss this issue in detail in 
Section~\ref{sect:WKB}.
It is important to note that our choice of
$F_{0,2}(z,z)$ differs from the definition
given in \cite{GS}.
\end{rem}

From now on, we restrict ourselves
to the case of degree $2$ covering
$\pi:\Sigma_s\lrar C$. This restriction is
necessary due to several
technical reasons. 
Since the spectral curve $\Sigma_s$ is a degree $2$ 
covering, we have $R_s = \Sigma_s\cap C\subset T^*C$,
and the Galois conjugation $\sigma$ is global on
$\Sigma_s$, which is the same as the $(-1)$ involution
$$
\sigma: T^*C\lrar T^*C.
$$
In particular, 
\begin{equation}
\label{eq:sigma eta}
\sigma^*\eta = -\eta.
\end{equation}
We denote $\sigma_p = \sigma$, and drop the
reference point $p$ from the recursion kernel, because
it does not depend on the ramification point any more.
The following lemma indicates how we
calculate the residues in the integration formulas.

\begin{lem}
\label{lem:W11}
We calculate
\begin{equation}
\label{eq:W01calc}
W_{1,1}(z_1) = \frac{B_{\Sigma_s}(z_1,\sigma(z_1))}
{2\eta(z_1)}
\in H^0\big(\Sigma_s,\Omega_{\Sigma_s}^1\tensor 
\cO_{\Sigma_s}(4R_s)\big).
\end{equation}
\end{lem}

\begin{rem}
Since our geometric setting is exactly the same,
 it is not surprising that the same
formula appears in \cite{KZ}, though
for a different purpose.
\end{rem}

\begin{proof}
Taking the advantage of (\ref{eq:kernel1}) and
(\ref{eq:sigma eta}),
 let us first identify the poles of the differential
form
$$
-\frac{\omega_s^{\sigma(z)-z}(z_1)}
{2\eta(z)} B_{\Sigma_s}(z,\sigma(z))
$$
in $z$, where $z_1\in \Sigma_s$ is a point arbitrarily 
chosen and fixed. We see that $z=p$ for every 
$p\in R_s$ is a pole, since $\eta$ vanishes on $R_s$. 
The fundamental form $B_{\Sigma_s}(z,z_1)$ has 
poles only along the diagonal, thus 
$B_{\Sigma_s}(z,\sigma(z))$ also has poles at $R_s$. 
Besides $R_s$, the form has simple poles at
$z=z_1$ and $z=\sigma(z_1)$. Since these are the
only poles, and remembering that
the integration variable is $z$, we use the
Cauchy integration formula to calculate
\begin{align*}
W_{1,1}(z_1) 
&= 
\half\;
\frac{1}{2\pi i}\sum_{p\in R_s}
\oint_{\gam_p} K(z,z_1) B_{\Sigma_s}(z,\sigma(z))
\\
&=
\half\;
\frac{1}{2\pi i}
\oint_{\gam_{z_1}\cup
\gam_{\sigma(z_1)} }
\frac{\omega_s^{\sigma(z)-z}(z_1)}
{2\eta(z)} B_{\Sigma_s}(z,\sigma(z))
\\
&=
\half\; \left(
-\frac{B_{\Sigma_s}(z_1,\sigma(z_1))}{2\eta(\sigma(z_1))}
+\frac{B_{\Sigma_s}(z_1,\sigma(z_1))}{2\eta(z_1)}
\right)
\\
&=
\frac{B_{\Sigma_s}(z_1,\sigma(z_1))}{2\eta(z_1)}.
\end{align*}
It is important
to note that $W_{1,1}(z_1)$ has poles only
at the ramification divisor $R_s$.
\end{proof}

It is clear from the above example that integration 
against $\omega_s^{\sigma(z)-z}(z_1)$ is 
exactly the Cauchy integration formula. 
Similarly, integration against 
$B_{\Sigma_s}(z_1,z_2)$ is the differentiation. Let
$f(z_1)$ be a meromorphic function on 
$\Sigma_s$. Then we have
\begin{equation}
\label{eq:B int}
\frac{1}{2\pi i}\oint_{\gam_{z_2}}f(z_1)
B_{\Sigma_s}(z_1,z_2)
=d_2 f(z_2),
\end{equation}
where the integration is taken with respect to the 
variable $z_1$. We note that the result is a
meromorphic $1$-form on $\Sigma_s$. 

\begin{lem}
\label{lem:W03}
We have
\begin{multline}
\label{eq:W03calc}
W_{0,3}(z_1,z_2,z_3)
=
\frac{1}{2\eta(z_1)}
\bigg(
B_{\Sigma_s}(z_1,z_2)B_{\Sigma_s}(z_1,\sigma(z_3))
+
B_{\Sigma_s}(z_1,z_3)B_{\Sigma_s}(z_1,\sigma(z_2))
\bigg)
\\
+
d_2\left(
\frac{\omega_s^{\sigma(z_2)-z_2}(z_1)
B_{\Sigma_s}(z_2,\sigma(z_3))}{2\eta(z_2)}
\right)
+
d_3\left(
\frac{\omega_s^{\sigma(z_3)-z_3}(z_1)
B_{\Sigma_s}(z_2,\sigma(z_3))}{2\eta(z_3)}
\right).
\end{multline}
\end{lem}

\begin{proof}
This time the change of contour $\sqcup_{p\in R_s}
\gam_p$ to other poles picks up contributions
from $z=z_i$ and $z=\sigma(z_i)$ for $i=1,2,3$.
As in the previous case, the contributions from 
$z=z_i$ and $z=\sigma(z_i)$ are always exactly
the same, 
which are compensated by the overall factor $1/2$. 
Then the calculations are performed at each
pole. For simple poles we 
use the Cauchy integration formula
with respect to 
$\omega_s^{\sigma(z)-z}(z_1)$, which
produces the first line of 
(\ref{eq:W03calc}). The second
line comes from the 
double poles of the Riemann fundamental
form, as explained in (\ref{eq:B int}).
\end{proof}

In terms of the local coordinate $z$ of 
(\ref{eq:local omega})-(\ref{eq:local sigma}),
we can approximate that $h(z)=z^2$. Then we 
have 
$$
W_{0,3}(z_1,z_2,z_3) = -\frac{dz_1dz_2dz_3}
{z_1^2 z_2^2 z_3^2} + O(1)dz_1dz_2dz_3.
$$
It is surprising that $W_{0,3}(z_1,z_2,z_3)$ has
poles  only at $z_i=p\in R_s$ for $i=1,2,3$, and
not along any diagonals.

\begin{thm}
\label{thm:Fgn recursion}
For $2g-2+n\ge 2$, the free energies satisfy the
following differential recursion formula:
\begin{multline}
\label{eq:Fgn recursion}
d_1 F_{g,n}(z_1,\dots,z_n)
\\
= -\sum_{j=2}^n 
\left[
\frac{\omega_s^{z_j-\sigma(z_j)}
(z_1)}{2\eta(z_1)}\cdot 
d_1F_{g,n-1}\big(z_{[\hat{j}]}\big)
-
\frac{\omega_s^{z_j-\sigma(z_j)}
(z_1)}{2\eta(z_j)}\cdot 
d_jF_{g,n-1}\big(z_{[\hat{1}]}\big)
\right]
\\
-\frac{1}{2\eta(z_1)}
d_{u_1}d_{u_2}
\left.
\left[F_{g-1,n+1}
\big(u_1,u_2,z_{[\hat{1}]}\big)
+\sum_{\substack{g_1+g_2=g\\
I\sqcup J=[\hat{1}]}}^{\text{stable}}
F_{g_1,|I|+1}(u_1,z_I)F_{g_2,|J|+1}(u_2,z_J)
\right]
\right|_{\substack{u_1=z_1\\u_2=z_1}}.
\end{multline}
\end{thm}

\begin{rem}
It has to be emphasized that (\ref{eq:Fgn recursion})
is given in terms of the exterior differentiation 
and contraction operations so that the equation
is indeed coordinate independent. 
The labels $z_1,\dots,z_n$ are simply indicating
which factor of the product $\Sigma_s^n$ the 
operation is taking place. They are not a
\emph{coordinate} of the spectral curve.
\end{rem}

\begin{rem}
Although we do not specify the $s\in V$
dependence of $F_{g,n}$ in the 
formula, (\ref{eq:Fgn recursion})
holds for the family of functions
$\{F_{g,n}^s\}_{s\in V}$.
\end{rem}

\begin{proof}
We wish to derive (\ref{eq:Wgn}) from 
(\ref{eq:Fgn recursion}).
We first recall the basic relations 
$$
d_z \omega^{z-b}_s(z_1) = B_{\Sigma_s}(z,z_1)
\qquad \text{and}\qquad
\omega^{z-b}_s(z_1)+\omega^{b-a}_s(z_1)
=\omega^{z-a}_s(z_1).
$$
Next let us apply the differentiation $d_2\cdots d_n$
everywhere in (\ref{eq:Fgn recursion}). The result is
\begin{multline}
\label{eq:Wgn intermediate}
W_{g,n}(z_1,\dots,z_n)
\\
=
-\sum_{j=2}^n \left[
\frac{1}{2\eta(z_1)}
\bigg(W_{0,2}(z_1,z_j)-W_{0,2}\big(z_1,\sigma(z_j)
\big)\bigg)W_{g,n-1}\big(z_{[\hat{j}]}\big)
\right]
\\
-\sum_{j=2}^n d_j
\left[
\frac{1}{2\eta(z_j)}
\omega_s^{\sigma(z_j)-z_j}W_{g,n-1}\big(z_{[\hat{1}]}\big)
\right]
\\
-\frac{1}{2\eta(z_1)}
\left[W_{g-1,n+1}
\big(u_1,u_2,z_{[\hat{1}]}\big)
\bigg|_{\substack{u_1=z_1\\u_2=z_1}}
+\sum_{\substack{g_1+g_2=g\\
I\sqcup J=[\hat{1}]}}^{\text{stable}}
W_{g_1,|I|+1}(z_1,z_I)W_{g_2,|J|+1}(z_1,z_J)
\right].
\end{multline}

It is time to evaluate the
residue integration in (\ref{eq:Wgn})
for $2g-2+n>1$. 
First we change
the integration contour from 
$\sum_{p\in R_s}\oint_{\gam_p}$ 
to the diagonals $z=z_j$ and $z=\sigma(z_j)$
for $j=1,2,\dots,n$. We can do this, because of
Proposition~\ref{prop:Wgn properties}, we know
that
$W_{g,n}$ has poles
only at $R_s$ for $2g-2+n>0$.
As noted in the example calculations 
Lemma~\ref{lem:W11} and Lemma~\ref{lem:W03}
above, the residue
contributions from $z=z_i$ and $z=\sigma(z_i)$
are always the same, and are compensated by
the overall factor of $1/2$ in the formula.
Thus we have
\begin{multline}
\label{eq:Wgn intermediate 2}
W_{g,n}(z_1,\dots,z_n)
=
\frac{1}{2\pi \sqrt{-1}}\sum_{i=1}^n
\oint_{\gam_{z_i}} \frac{\omega_s^{\sigma(z)-z}
(z_1)}{2\eta(z)}
\\
\times
\Bigg[
\sum_{j=2}^n \bigg(
W_{0,2}(z,z_j) 
W_{g,n-1}(\sigma(z),z_{[\hat{1},\hat{j}]})
+
W_{0,2}(\sigma(z),z_j) 
W_{g,n-1}(z,z_{[\hat{1},\hat{j}]})
\bigg)
\\
+
W_{g-1,n+1}\big(z,\sigma(z),z_{[\hat{1}]}\big)
+
\sum_{\substack{g_1+g_2=g
\\
I\sqcup J = \{2,\dots,n\}}}
^{\text{stable}}
W_{g_1,|I|+1}(z,z_I)W_{g_2,|J|+1}(\sigma(z),z_J)
\Bigg].
\end{multline}
The contribution from the integration
around $z=z_1$ comes from the simple
pole of the differential form
$\omega_s^{\sigma(z)-z}
(z_1)$. The integration is done by 
the Cauchy integration formula, and the result is
\begin{multline*}
-\frac{1}{2\eta(z_1)}
\sum_{j=2}^n 
\bigg(
W_{0,2}(z_1,z_j) 
-
W_{0,2}\big(z_1,\sigma(z_j)\big) 
\bigg)
W_{g,n-1}(z_{[\hat{j}]})
\\
-\frac{1}{2\eta(z_1)}
\left[
W_{g-1,n+1}\big(z_1,z_1,z_{[\hat{1}]}\big)
+
\sum_{\substack{g_1+g_2=g
\\
I\sqcup J = \{2,\dots,n\}}}
^{\text{stable}}
W_{g_1,|I|+1}(z_1,z_I)
W_{g_2,|J|+1}\big(\sigma(z_1),z_J\big)
\right].
\end{multline*}
Here we have used (\ref{eq:Wgn balance}).
We have thus recovered
 the first and the third lines of
the right-hand side of (\ref{eq:Wgn intermediate}).

The contribution in (\ref{eq:Wgn intermediate 2})
from the integration around  $z=z_j$, $j\ge 2$, comes
from the diagonal double poles of $W_{0,2}(z,z_j)$.
Since $W_{0,2} = B_{\Sigma_s}$ acts as the differentiation
kernel (\ref{eq:B int}), it is easy to see that
the result is exactly the same as the 
second line of the right-hand side of 
(\ref{eq:Wgn intermediate}).
This completes the proof.
\end{proof}

\section{The $\lam$-connections and the
WKB method}
\label{sect:lam}

The precise notion we need to describe our quantum 
curve is Deligne's $\lam$-connection, where
$\lam$ is  a formal parameter.
In physics the notation $\lam=\hbar$ is commonly
used. Since the literature on quantum curves
consistently
use the Planck constant notation, we adopt it here
as well.
In this section we review the
materials on $\lam$-connections that we need in
this paper, following
 the excellent article
of Arinkin \cite{A}. In what follows, when
we say an $\hbar$-connection,  we are
indeed referring to a $\lam$-connection with 
$\lam=\hbar$.
The most important feature of the 
$\hbar$-connections is that the 
WKB approximation method can be 
applied to this type of connections.

\begin{Def}[$\hbar$-Connection]
Let $(E,\phi)$ be a Higgs pair defined on $C$. 
An $\hbar$-\textbf{connection} on $E$ associated
with the pair $(E,\phi)$ is a 
$\bC$-linear homomorphism
$$
\nabla^\hbar :E\lrar E\tensor \Omega_C^1
$$
subject to the following two conditions:
\begin{equation}
\label{eq:hbar connection}
\nabla^\hbar (f\cdot v) = f\cdot \nabla^\hbar (v)
+  v\tensor (\hbar \;df)
\end{equation}
for $f\in \cO_C$ and $v\in E$, and 
\begin{equation}
\label{eq:hbar = 0}
\phi = \nabla^\hbar \big|_{\hbar=0} .
\end{equation}
For every tangent vector  $X\in T_xC$  at 
$x\in C$, the $\bC$-linear 
$\hbar$-covariant derivative
$$
\nabla_X^\hbar :E\lrar E
$$
is defined by the derivation equation
\begin{equation}
\label{eq:hbar derivation}
\nabla_X^\hbar(f\cdot v) = f\cdot 
\nabla_X^\hbar(v) + \hbar X(f) \cdot v.
\end{equation}
\end{Def}

If $\hbar\ne 0$, then $\frac{1}{\hbar}\nabla^\hbar$
is a holomorphic connection in $E$. Hence $E$ is
flat, and it necessarily has $\deg(E) = 0$.

We consider the variable $\hbar$ as a deformation
parameter. First we extend the base curve $C$
to a formal family 
\begin{equation}
\label{eq:formal}
C[[\hbar]] := \lim_{\substack{\lrar \\ n}} 
C\times \Spec\left(
\bC[\hbar]/(\hbar^n)\right).
\end{equation}
A $\bC[[\hbar]]$-linear
 $\hbar$-connection on a vector bundle $E$
over $C[[\hbar]]$ is defined in the same way 
as above. As a flat connection on a vector bundle
makes the bundle a $D$-module, an $\hbar$-connection
on $C[[\hbar]]$ gives $E$ a $D$-module structure.
Since we do not consider differentiations with respect
to $\hbar$, we call a vector bundle with a
$\bC[[\hbar]]$-linear
 $\hbar$-connection a $D^\hbar$\textbf{-module}.

A $D$-module on a complex manifold $M$
gives rise to a characteristic variety in $T^*M$.
When the $D$-module is holonomic, the
characteristic variety becomes a Lagrangian
in $T^*M$. For our case, 
any $D$-module over a complete algebraic
curve $C$ is holonomic, and defines a
Lagrangian subvariety in $T^*C$. 
These Lagrangians are either the $0$-section 
of the cotangent bundle $T^*C$, or a union of
finite
number of fibers. They satisfy the 
$\bC^*$-invariance with respect to the 
$\bC^*$-action on $T^*C$.
The spectral curves we consider (\ref{eq:family})
are not those Lagrangians as the characteristic 
variety of a $D$-module. They do not 
satisfy the $\bC^*$-invariance. 

The sheaf of $\hbar$-differential 
operators $\cD^\hbar$ on $C[[\hbar]]$ is
constructed by gluing 
\begin{equation}
\label{eq:DU}
\cD^\hbar\big|_{U[[\hbar]]} = 
\cO_{U[[\hbar]]}\left[\hbar \frac{d}{dx}\right],
\end{equation}
where $x$ is a coordinate of an affine
open subscheme $U$ of $C$. 
The \textbf{classical limit}
of a $D^\hbar$-module
 is the mod $\hbar$-reduction,
which simply is an $\cO_C$-module.
The passage between the spectral curves
of Hitchin fibrations and $D$-modules is
not the classical limit, or the characteristic
variety. It is the semi-classical limit,
and it requires the WKB method (see for example,
\cite{BO})
to define.

Let $(E,\nabla^\hbar)$ be a 
 $\bC[[\hbar]]$-linear
 $\hbar$-connection on a vector bundle $E$
over $C[[\hbar]]$. As a $D^\hbar$-module, 
it is easy to show that on an affine 
open $U\subset C$ we have a
differential operator $P(x,\hbar)
\in \cD^\hbar\big|_{U[[\hbar]]}$ such that 
\begin{equation}
\label{eq:D/DP}
E|_{U[[\hbar]]} \isom 
\left.\left(\cD^\hbar\big/
\cD^\hbar P\right)\right|_{U[[\hbar]]}.
\end{equation}
Usually we consider a solution of 
\begin{equation}
\label{eq:PPsi}
P(x,\hbar)\Psi(x,\hbar) = 0
\end{equation}
as an element 
$$
\Psi(x,\hbar)\in 
\Hom\big(E_{U[[\hbar]]},\cO_{U[[\hbar]]}\big).
$$
The WKB method is a mechanism to construct
the solution of (\ref{eq:PPsi}) that \emph{does not}
have a convergent limit as $\hbar\rar 0$,
by the 
\emph{singular perturbation method}
\begin{equation}
\label{eq:WKB}
\Psi(x,\hbar) =\exp\left(\sum_{m=0}^\infty
\hbar^{m-1} S_m(x)\right).
\end{equation}
Here $S_m(x)$ is a holomorphic function
defined on an open subset $U\subset C$,
but has poles at certain points of $C$. 
The parameter $\hbar$ is considered to be  
small, so the $m=0$ contribution
is \emph{singular}. The equation (\ref{eq:PPsi})
is interpreted as 
\begin{equation}
\label{eq:PPsi2}
\bigg(
e^{-\frac{1}{\hbar}S_0(x)}
P(x,\hbar)
e^{\frac{1}{\hbar}S_0(x)}
\bigg)
\exp\left(\sum_{m=1}^\infty
\hbar^{m-1} S_m(x)\right)=0.
\end{equation}
Since 
$$
P(x,\hbar)\in \cO_{U[[\hbar]]}\left[
\hbar\frac{d}{dx}\right],
$$
both the operator and the solution of
(\ref{eq:PPsi2}) are defined
over $U[[\hbar]]$.

\begin{Def}
\label{def:SCL}
Consider an operator $P(x,\hbar)$ 
defined on an open subset $U\subset C$ that is in
the normal ordering expression
\begin{equation}
\label{eq:normal}
P(x,\hbar) = \sum_{k=0}^n
a_k(x,\hbar)\left(\hbar \frac{d}{dx}\right)^{n-k},
\end{equation}
where $a_k(x,\hbar)\in \cO_{U[[\hbar]]}$.
Then we have
\begin{equation}
\label{eq:SCL1}
e^{-\frac{1}{\hbar}S_0(x)}
P(x,\hbar)
e^{\frac{1}{\hbar}S_0(x)}
\bigg|_{\hbar=0}
=
\sum_{k=0}^n
a_k(x,0) \big(S_0'(x)\big)^{n-k},
\end{equation}
where $'$ indicates the $x$-derivative. 
The \textbf{semi-classical limit} of 
the differential equation (\ref{eq:PPsi})
at $\hbar=0$
is the formula (\ref{eq:SCL1}).
If we use an indeterminate 
$y=S_0'(x)$, then the
semi-classical limit is the mod $\hbar$-reduction
\begin{equation}
\label{eq:total symbol}
\sum_{k=0}^n
a_k(x,0) y^{n-k}
\end{equation}
of the \textbf{total symbol}
of the normal ordered operator (\ref{eq:normal}).
\end{Def}

Note that the semi-classical limit 
(\ref{eq:total symbol}) 
is neither 
the \emph{symbol} 
nor the \emph{characteristic variety}
of the operator 
$P(x,\hbar)$. 
The passage from (\ref{eq:total symbol}) to 
(\ref{eq:normal}) is the \textbf{quantization} 
we are discussing in this paper.
In an abstract setting, of course there is 
no way determining a differential operator
from its total symbol (\ref{eq:total symbol})
at $\hbar=0$. 
In the next section we show that a 
$SL(2,\bC)$-Hitchin spectral curve has a 
\emph{unique} quantization.

\section{The WKB approximation
and  quantum curves}
\label{sect:WKB}

We are now ready to state and prove the 
main theorem of this paper.

\begin{thm}
\label{thm:precise}
Let $\cH_C(2,0)_0$ denote the modui stack of
rank $2$ Higgs pairs of degree $0$ vector
bundles with a fixed determinant line bundle, and
consider the $SL(2,\bC)$-Hitchin fibration
\begin{equation}
\label{eq:SL2Hitchin}
\mu_H:\cH_C(2,0)_0\lrar V_{SL}^* :=
H^0\big(C,(\Omega_C^1)^{\tensor 2}\big).
\end{equation}
For a generic spectral data $s\in V_{SL}^*$, 
there is a contractible open neighborhood 
$s\in V\subset V_{SL}^*$ such that 
the \textbf{family} of 
smooth spectral curves 
$$
\widetilde{\Sigma}\big|_V=
\{\Sigma_s\}_{s\in V}
$$ is quantizable via the WKB method.
\end{thm}

\begin{rem}
The most involved technical part of this paper
is the reduction of the differential
recursion (\ref{eq:Fgn recursion}) into
an ordinary differential equation via
the \emph{principal specialization}
\begin{equation}
\label{eq:principal}
z_1=z_2=\cdots=z_n=z.
\end{equation}
We note that for the case of simple and
double Hurwitz numbers and related 
topics discussed in \cite{BHLM,  MSS, MS, Zhou4},
the principal specialization corresponds
to the reduction of a summation over all
Young diagrams (or partitions) into 
a sum over $1$-row Young diagrams. 
Thus the formulas dramatically 
simplify, and this is the key to constructing
the quantum curves. 
For the case of Hitchin fibrations we do not
 have an interpretation as a sum over partitions,
 and the process of principal specialization
 becomes technically more difficult. 
 \end{rem}
 
 \begin{rem}
 The $s\in V$ dependence does not 
 pose any difficulty, because the only 
 consideration we need is the consistent
 integration we have taken care of in 
 Section~\ref{sect:integral} for the choice 
 of the subset $V$ with a consistent 
 symplectic basis for $H_1(\Sigma_s,\bZ)$. 
 The calculations in this section are thus
 all carried out over this family.
 \end{rem}

We first recall a trivial lemma from \cite{MS}:

\begin{lem}
Let $f(z_1,\dots,z_n)$ be a symmetric function
in $n$ variables. 
Then
\begin{equation}
\begin{aligned}
\label{eq:dfdt}
\frac{d}{dz}f(z,z,\dots,z) 
&= n
\left.
\left[
\frac{\partial}{\partial u}f(u,z,\dots,z)
\right]
\right|_{u=z};
\\
\frac{d^2}{dz^2}f(z,z,\dots,z) 
&=
n\left.\left[
\frac{\partial^2}{\partial u ^2} 
f(u,z,\dots,z)
\right]\right|_{u=z}
\\
&\qquad
+
n(n-1)
\left.\left[
\frac{\partial^2}{\partial u_1 \partial u_2} 
f(u_1,u_2,z,\dots,z)
\right]\right|_{u_1=u_2=z}.
\end{aligned}
\end{equation}
For a function in one variable 
$f(z)$,
we have
\begin{equation}
\label{eq:lhopital}
\lim_{z_2\rar z_1}
\left[
\omega^{z_2-b}(z_1)\big(
f(z_1)-f(z_2)
\big)
\right]
=d_1f(z_1),
\end{equation}
where $\omega^{z_2-b}(z_1)$ is the $1$-form
of {\rm{(\ref{eq:B and omega})}}.
\end{lem}

The rest of the section is devoted to proving
Theorem~\ref{thm:precise}.

\begin{proof}[Proof of Theorem~\ref{thm:precise}]
For the purpose of calculation, let us choose one
of the ramification points $p\in R$ of
the covering $\pi:\Sigma_s\lrar C$ for a generic 
spectral data $s = s_2\in 
H^0\big(C,(\Omega_C^1)^{\tensor 2}\big)$,
and assume that all points $z_1,\dots,z_n$
are close to $p$, but not quite equal. 
As a consequence, their Galois conjugates
$\sigma(z_j)$'s are also close to $p$. 
On a neighborhood we choose a local coordinate $z$
around $p$ such that $z=0$ defines $p$ and
that $\sigma(z) = -z$. We use
the local expressions (\ref{eq:local omega}),
(\ref{eq:local B}), (\ref{eq:local eta}),
and the relations (\ref{eq:BK involution}).
Using the notation $\partial_z = \partial/\partial z$,
we have a local formula equivalent to 
(\ref{eq:Fgn recursion}) that is valid
for $2g-2+n\ge 2$:
\begin{multline}
\label{eq:Fgn local}
\partial_{z_1} F_{g,n}(z_1,\dots,z_n)
\\
= -\sum_{j=2}^n 
\left[
\frac{\omega_s^{z_j-\sigma(z_j)}
(z_1)}{2h(z_1)dz_1}\cdot 
\partial_{z_1}F_{g,n-1}\big(z_{[\hat{j}]}\big)
-
\frac{\omega_s^{z_j-\sigma(z_j)}
(z_1)}{dz_1\cdot 2h(z_j)}\cdot 
\partial_{z_j}F_{g,n-1}\big(z_{[\hat{1}]}\big)
\right]
\\
-\frac{1}{2h(z_1)}
\frac{\partial^2}{\partial u_1\partial u_2}
\left.
\left[F_{g-1,n+1}
\big(u_1,u_2,z_{[\hat{1}]}\big)
+\sum_{\substack{g_1+g_2=g\\
I\sqcup J=[\hat{1}]}}^{\text{stable}}
F_{g_1,|I|+1}(u_1,z_I)F_{g_2,|J|+1}(u_2,z_J)
\right]
\right|_{\substack{u_1=z_1\\u_2=z_1}}.
\end{multline}
Let us apply (\ref{eq:principal}).
The left-hand side becomes $\frac{1}{n}
\partial_z F_{g,n}(z,\dots,z)$.
To calculate the contributions from 
the first line of the right-hand side of 
(\ref{eq:Fgn local}),
we choose $j>1$ and set $z_i=z$ for all $i$
except for $i=1,j$. Then take the limit $z_j\rar z_1$. 
In this procedure, we note that the contributions
from the simple pole 
of $\omega_s^{z_j-\sigma(z_j)}(z_1)$
at $z_1 = \sigma(z_j)$ cancel at $z_1=z_j$.
Thus we obtain
\begin{multline*}
\left.
-\sum_{j=2}^n \frac{1}{z_1-z_j}
\left(
\frac{1}{2h(z_1)}\partial_{z_1}F_{g,n-1}
(z_1,z,\dots,z)
-\frac{1}{2h(z_j)}\partial_{z_j}F_{g,n-1}
(z_j,z,\dots,z)
\right)
\right|_{z_1=z_j}
\\
=
-\sum_{j=2}^n
\partial_{z_1}
\left(
\frac{1}{2h(z_1)}\partial_{z_1}F_{g,n-1}
(z_1,z,\dots,z)
\right)
\\
=
-(n-1)\partial_{z_1}
\left(
\frac{1}{2h(z_1)}\partial_{z_1}F_{g,n-1}
(z_1,z,\dots,z)
\right)
\\
=
-(n-1)\partial_{z_1}
\left(
\frac{1}{2h(z_1)}
\right)\partial_{z_1}F_{g,n-1}
(z_1,z,\dots,z)
-
\frac{n-1}{2h(z_1)}
\partial_{z_1}^2F_{g,n-1}
(z_1,z,\dots,z).
\end{multline*}
The limit $z_1\rar z$ then produces
\begin{multline}
\label{eq:unstable}
-\partial_z \frac{1}{2h(z)}\cdot
\partial_z F_{g,n-1}(z,\dots,z)
-
\frac{1}{2h(z)}\partial_z^2 F_{g,n-1}(z\dots,z)
\\
+
\left.
\frac{(n-1)(n-2)}{2h(z)}\frac{\partial^2}
{\partial u_1\partial u_2}
F_{g,n-1}(u_1,u_2,z\dots,z)
\right|_{u_1=u_2=z}.
\end{multline}
To calculate the principal specialization of
the second line of the right-hand side
of (\ref{eq:Fgn local}),
we note that since all points $z_i$'s for $i\ge 2$
are set to be equal,  
a  set partition by index sets $I$ and $J$ 
becomes a partition of $n-1$ with a combinatorial 
factor that counts the redundancy.
The result is
\begin{multline}
\label{eq:stable}
-\left.
\frac{1}{2h(z)}\frac{\partial^2}
{\partial u_1\partial u_2}
F_{g-1,n+1}(u_1,u_2,z\dots,z)
\right|_{u_1=u_2=z}
\\
-\frac{1}{2h(z)}
\sum_{\substack{g_1+g_2=g\\
n_1+n_2=n-1}}^{\text{stable}}
\partial_z
F_{g_1,n_1+1}(z,\dots,z)\cdot
\partial_zF_{g_2,n_2+1}(z_,\dots,z).
\end{multline}
Assembling (\ref{eq:unstable}) and (\ref{eq:stable})
together, we obtain
\begin{multline}
\label{eq:ps1}
\frac{1}{2h(z)}
\left[
\partial_z^2 F_{g,n-1}(z\dots,z)
+
\sum_{\substack{g_1+g_2=g\\
n_1+n_2=n-1}}^{\text{stable}}
\partial_z
F_{g_1,n_1+1}(z,\dots,z)\cdot
\partial_zF_{g_2,n_2+1}(z_,\dots,z)
\right]
\\
+
\frac{1}{n}\partial_zF_{g,n}(z,\dots,z)
+
\partial_z \frac{1}{2h(z)}\cdot
\partial_z F_{g,n-1}(z,\dots,z)
\\
=
\left.
\frac{(n-1)(n-2)}{2h(z)}\frac{\partial^2}
{\partial u_1\partial u_2}
F_{g,n-1}(u_1,u_2,z\dots,z)
\right|_{u_1=u_2=z}
\\
-\left.
\frac{1}{2h(z)}\frac{\partial^2}
{\partial u_1\partial u_2}
F_{g-1,n+1}(u_1,u_2,z\dots,z)
\right|_{u_1=u_2=z}.
\end{multline}

Following the construction of the quantum
curves of \cite{GS, MS}, we now 
apply the operation
$\sum_{2g-2+n=m}\frac{1}{(n-1)!}$ to 
(\ref{eq:ps1}) above, and write the result in
terms of 
\begin{equation}
\label{eq:Sm}
S_m(z) := \sum_{2g-2+n=m-1}\frac{1}{n!}
F_{g,n}(z,\dots,z),
\end{equation}
to fit into the WKB formalism.
For $m\ge 2$, $S_m(z)$ is a meromorphic 
function on $\Sigma_s$ with a pole at each
ramification point (Lagrangian singularity)
$p\in R_s$ of order $3m-3$. This can be easily
seen by the fact that 
$F_{g,n}(z,\dots,z)$ has
a pole  of order $6g-6+3n$ 
at each $p\in R_s$. And this fact
follows by induction from the 
integral recursion (\ref{eq:Fgn recursion})
on $F_{g,n}$,
and the initial conditions (\ref{eq:W11})
and (\ref{eq:W03}).

Our first remark is that summing over
 all possibilities
of $(g,n)$ with the fixed value of $2g-2+n$, 
the right-hand side of (\ref{eq:ps1}) becomes $0$.
Thus we have established

\begin{thm}
The functions $S_m(z)$ of {\rm{(\ref{eq:Sm})}}
for $m\ge 2$
satisfy the recursion formula
\begin{equation}
\label{eq:Sm recursion}
\frac{1}{2h(z)}
\left(
\frac{d^2S_m}{dz^2}+
\sum_{\substack{a+b=m+1\\a,b\ge 2}}
\frac{dS_a}{dz} \;\frac{dS_b}{dz}
\right)
+
\frac{dS_{m+1}}{dz}
+
\frac{d}{dz}
\left(\frac{1}{2h(z)}\right) \frac{dS_m}{dz}
=0.
\end{equation}
It can also be written as a coordinate-free manner
as an equation for meromorphic $1$-forms on 
$\Sigma_s$:
\begin{equation}
\label{eq:Sm coordinate free}
dS_{m+1}+
\frac{1}{2\eta}
\sum_{\substack{a+b=m+1\\a,b\ge 2}}
dS_a\cdot dS_b +
d\left(\frac{1}{2\eta}dS_m\right)= 0,
\end{equation}
where $1/\eta$ is again the contraction operator
with respect to the $1$-form $\eta$.
\end{thm}

Recall the local geometry of
the spectral curve
$$
p\in \Sigma_s\subset T^*C,
$$
and that $p\in C$ is also on the $0$-section 
of the cotangent bundle $T^*C$.  
We trivialize the cotangent bundle near $x=p$,
where $x$ is a local coordinate on $C$, and
let $y$ be the fiber coordinate of $T^*_xC$. 
The relation between $(x,y)\in T^*C$ and
the local coordinate $z$ of $\Sigma_s$ around 
$p\in \Sigma_s$ is given by the formula
\begin{equation}
\label{eq:eta = ydx}
\eta = h(z)dz = ydx.
\end{equation}
Let the local expression of the spectral data $s=s_2$ 
be $s_2 = s_2(x)(dx)^2$. 
Then the equation for the spectral curve $\Sigma_s$
near $p\in \Sigma_s$ is given by
\begin{equation}
\label{eq:local Cs}
y^2+s_2(x) = 0.
\end{equation}
The local expression of the quantum curve,
which is an $\hbar$-differential operator,
becomes 
\begin{equation}
\label{eq:local qc}
P(x,\hbar):=\hbar^2\left(\frac{d}{dx}\right)^2 +s_2(x).
\end{equation}
Following the method of
Berg\`ere-Eynard \cite{EB} and the WKB formalism
 of Gukov-Su\l kowski 
\cite{GS}, we
define
\begin{align}
\label{eq:F}
F(z,\hbar) &= \sum_{m=0}^\infty
\hbar^{m-1} S_m (z)
=\sum_{g\ge 0}\sum_{n\ge 1}
\hbar^{2g-2+n}\frac{1}{n!} F_{g,n}(z,\dots,z),
\\
\label{eq:Psi}
\Psi(z,\hbar) &= \exp F(z,\hbar).
\end{align}
The truncated summation for $m\ge 2$ in
(\ref{eq:F}), and the corresponding
portion of  (\ref{eq:Psi}), are  functions on 
$C[[\hbar]]$ with essential singularities
at each Lagrangian singularity of the 
spectral curve $\pi:\Sigma_s\lrar C$.
The factor $e^{\frac{1}{\hbar}S_0}$
in $\Psi$ plays the role of determining the 
semi-classical limit, as explained in 
Section~{\ref{sect:lam}}.

Using (\ref{eq:eta = ydx}) we identify the
derivation 
\begin{equation}
\label{eq:d/dx}
\frac{d}{dx} = \frac{y}{h(z)}\;\frac{d}{dz},
\end{equation}
which comes from the push-forward $\pi_*(d/dz)$.
The transformation (\ref{eq:d/dx}) is
singular at every ramification point.
The Schr\"odinger equation is calculated as
\begin{align}
\label{eq:Sch1}
&P(x,\hbar)\Psi(z,\hbar) = 0
\\
\label{eq:Sch2}
\Longleftrightarrow\quad
&\hbar^2\left(
\frac{d^2F}{dx^2}+\frac{dF}{dx}\cdot
\frac{dF}{dx}\right) + s_2(x) = 0
\\
\label{eq:Sch3}
\Longleftrightarrow\quad
&\sum_{m=0}^\infty
\hbar^{m+1}\frac{d^2S_m}{dx^2}
+
\sum_{a,b\ge 0}\hbar^{a+b}
\frac{dS_a}{dx}\cdot \frac{dS_b}{dx}
+s_2(x) = 0.
\end{align}
Collecting the coefficient of the $\hbar^0$ terms
in (\ref{eq:Sch3}), we obtain the
semi-classical limit
\begin{equation}
\label{eq:SCL}
\left(\frac{dS_0}{dx}\right)^2 + s_2(x) = 0.
\end{equation}
From (\ref{eq:local Cs}) and (\ref{eq:SCL}) we
conclude that
\begin{equation}
\label{eq:S0}
\frac{dS_0}{dx} = y = \sqrt{-s_2(x)}.
\end{equation}
This is consistent with our choice of $W_{0,1}$
of the Eynard-Orantin theory 
(\ref{eq:W01}):
$$
dS_0 = dF_{0,1} = W_{0,1} = \eta = ydx.
$$
Moreover, if we allow terms $a=0$ or $b=0$ in
(\ref{eq:Sm recursion}), then what we have in 
addition is
$$
\frac{1}{2h(z)} 2\frac{dS_0}{dz}
\frac{dS_{m+1}}{dz} = \frac{1}{h(z)} 
\frac{h(z)}{y} \frac{dS_0}{dx}\frac{dS_{m+1}}{dz}
=\frac{dS_{m+1}}{dz}.
$$
In other words, the $\frac{dS_{m+1}}{dz}$ term 
already there in (\ref{eq:Sm recursion}) is absorbed 
in the split differentiation for $a=0$ and $b=0$.

Here we comment that $S_0 = \int \eta$
is not a function on $\Sigma_s$. Since $\eta$ is a 
holomorphic $1$-form on $\Sigma_s$, 
its integral is defined only on the
universal covering of $\Sigma_s$. From
 (\ref{eq:SCL}), we calculate the conjugated 
operator    (\ref{eq:PPsi2})
\begin{equation}
\label{eq:P0}
e^{-\frac{1}{\hbar}S_0}P(x,\hbar)
e^{\frac{1}{\hbar}S_0}
=\hbar^2\frac{d^2}{dx^2} +2 \hbar \frac{dS_0}{dx}
\frac{d}{dx} + \hbar \frac{d^2S_0}{dx^2}.
\end{equation}

The $\hbar^1$ terms of (\ref{eq:Sch3})
give what we call the \textbf{consistency
condition}
\begin{equation}
\label{eq:consistency}
\frac{d^2S_0}{dx^2} + 2\frac{dS_0}{dx}\cdot
\frac{dS_1}{dx} = 0,
\end{equation}
which also follows from (\ref{eq:P0}).
We recall that until now we have never 
\emph{defined} what we want to use
as $F_{0,2}(z_1,z_2)$.
The defining equation $d_1d_2F_{0,2} = W_{0,2}$
alone does not determine $F_{0,2}$ because
we can add terms 
$$
F_{0,2}(z_1,z_2) + f(z_1)+f(z_2)
$$ 
using an arbitrary function $f(z)$. The principal
specialization then becomes $F_{0,2}(z,z) + 2f(z)$,
which makes
$$
S_1 = \half F_{0,2}(z,z) + f(z).
$$
This situation allows us to \emph{define}
the quantity $S_1$ by a solution of the
consistency condition (\ref{eq:consistency}).
Thus we define, 
\begin{equation}
\label{eq:S1}
S_1 = \int^x\frac{dS_1}{dx}dx = -\half 
\log \frac{dS_0}{dx}.
\end{equation}
This makes
\begin{equation}
\label{eq:e^S1}
e^{S_1} = \frac{1}{\sqrt{y}}.
\end{equation}

\begin{rem}
We note that the choice we need to make for
$S_1$, the formula given in (\ref{eq:S1}),
is different from the choice of the
\emph{torsion} term of \cite{GS}. 
\end{rem}

More importantly for our purpose,
we read off from (\ref{eq:consistency})
that
\begin{equation}
\label{eq:S1'}
\frac{dS_1}{dx} =-
\half  \frac{\frac{d}{dx}\sqrt{-s_2(x)}}
{\sqrt{-s_2(x)}}.
\end{equation}
Note that $s_2(x)$ has a simple zero at
each branch point $p\in C$. If $x$ is chosen
as a local coordinate centered at $p$, then
(\ref{eq:S1'}) is a meromorphic
function with a simple pole at $p$.
The conjugation of (\ref{eq:P0})
by $e^{S_1}$ is calculated as
\begin{equation}
\label{eq:P1}
e^{-S_1}
e^{-\frac{1}{\hbar}S_0}P(x,\hbar)
e^{\frac{1}{\hbar}S_0} 
e^{S_1}
=\hbar^2\frac{d^2}{dx^2} 
+2 \left(\hbar \frac{dS_1}{dx}
+  \frac{dS_0}{dx}
\right)\hbar \frac{d}{dx} 
\in \cD^\hbar (U),
\end{equation}
where $U\subset C$ is an open subset that does
not contain any branch point of the covering
$\pi$.

Finally we have

\begin{lem} The consistency condition
{\rm{(\ref{eq:consistency})}} makes
{\rm{(\ref{eq:Sm recursion})}} and
{\rm{(\ref{eq:Sch3})}} equivalent
on any open subset $U\subset C$ that
is away from the caustics.
\end{lem}

\begin{proof}
First
we calculate the second differential operator,
from (\ref{eq:d/dx}) and (\ref{eq:S0}):
\begin{equation*}
\frac{d^2}{dx^2} = \frac{d}{dx}
\left(\frac{S_0'}{h}\frac{d}{dz}\right)
=
\frac{(S_0')^2}{h^2}\frac{d^2}{dz^2}
+\frac{S_0'}{h}\frac{d}{dz}\left(\frac{S_0'}{h}
\right) \cdot \frac{d}{dz},
\end{equation*}
denoting by $S_0'=dS_0/dx$.
The $\hbar^{m+1}$-terms of (\ref{eq:Sch3})
then produce
\begin{equation}
\label{eq:hbar m+1}
\frac{(S_0')^2}{h^2}
\left(
\frac{d^2}{dz^2} S_m +
\sum_{a+b=m+1} \frac{dS_a}{dz}\frac{dS_b}{dz}
\right)
+
\frac{S_0'}{h}\frac{d}{dz}\left(\frac{S_0'}{h}
\right) \cdot \frac{dS_m}{dz}
=0.
\end{equation}
The coefficients of  
$dS_m/dz$ in (\ref{eq:hbar m+1})
are
\begin{multline*}
2 \frac{(S_0')^2}{h^2}\frac{dS_1}{dz}
+\frac{S_0'}{h}\frac{d}{dz}\left(\frac{S_0'}{h}
\right)
=
2 \frac{(S_0')^2}{h^2}\frac{h}{S_0'} S_1'
+\frac{d}{dx}\left(\frac{S_0'}{h}
\right)
\\
=\frac{1}{h}
\big(
2S_0'S_1'+S_0''
\big)
+
S_0'\frac{d}{dx}\left(\frac{1}{h}\right)
=
S_0'\frac{d}{dx}\left(\frac{1}{h}\right)
\\
=
\frac{(S_0')^2}{h^2}\cdot 2h\frac{d}{dz}\left(
\frac{1}{2h}\right).
\end{multline*}
This is exactly what the last term of 
(\ref{eq:Sm recursion}) has, after
adjusting the overall coefficient of
$\frac{(S_0')^2}{h^2}\cdot 2h$. This complets
the proof of Lemma.
\end{proof}

With the above lemma, we have competed the proof of 
the main theorem.
\end{proof}

\begin{rem}
The Schr\"odinger equation (\ref{eq:Sch1})
has a holomorphic coefficient $s_2(x)$. Therefore,
the solution is also holomorphic. 
The expression (\ref{eq:Psi}) is therefore valid
only for points away from the caustics. In other
words, the WKB method is not valid at 
the caustics. 
The local behavior of $\Psi(z,\hbar)$ 
at every Lagrangian singularity is 
universal, because $s_2(x)$ has a simple zero
at each point $p\in R_s$ of the
caustics. Here recall that $R_s = \Sigma_s\cap C$,
so $R_s$ is also  the branch divisor in $C$.
If we have chosen a local coordinate $x$ of $C$
at $p\in R_s$ so that $x=0$ gives the point  $p$,
then on a small neighborhood of $p$ we
have an expression $s_2(x) =-x$.
Since the differential equation becomes
$$
\left(\hbar^2 \frac{d^2}{dx^2}-x\right)
\Psi(x,\hbar)=0,
$$
it is obvious that the local 
solution is given by the Airy function
(see for example, \cite{AF}).
This calculation has been carried out 
in \cite{A,EB}. The spectral curve in this 
case is locally $x=y^2$, for which the
Eynard-Orantin theory produces the
cotangent
$\psi$-class intersection numbers
considered by Witten \cite{W1991} and
Kontsevich \cite{K1992}. See for example,
\cite{DMSS}, on this connection.
\end{rem}

\hyphenation{Gar-ou-fal-i-dis}

\begin{ack}
The authors are grateful to 
Bertrand Eynard, for his continuous inspiration
and encouragements to them on the topics 
that led to the 
current work.
The authors also thank  
Albert Schwarz for his enlightening seminar talk
at UC Davis, which triggered them to work 
on the quantization of Hitchin spectral 
curves, and  Xiaojun Liu and
Axel Saenz for many discussions. M.M.\ thanks
Philip Boalch,
Boris Dubrovin,
Stavros Garoufalidis,
Tam\'as Hausel,
John Hunter,
Jerry Kaminker,
Maxim Kontsevich,
Sergei Lando,
Marta Mazzocco,
Paul Norbury,
Brad Safnuk,
Sergey Shadrin, and
Piotr Su\l kowski
for useful comments and discussions. 
The research of O.D.\ has been supported by
the Arthur J.~Krener fund in the University of 
California, Davis, and GRK 1463 \emph{Analysis,
Geometry, and String Theory} at the 
Leibniz Universit\"at 
 Hannover.
During the preparation of this
paper,
M.M.\ has received research  support  
from many institutions, which is gratefully
acknowledged: 
l'Institut des Hautes \'Etudes Scientifiques in
 Bures-sur-Yvette, 
Max-Planck Institut f\"ur Mathematik in Bonn,
la Scuola Internazionale Superiore di Studi Avanzati 
in Trieste, 
 the American Institute of Mathematics in Palo Alto,
 the Center for Quantum Geometry of Moduli 
 Spaces at  Aarhus University, 
 Instituto Superior T\'ecnico in Lisbon,
the University of Melbourne, 
and  Universiteit van Amsterdam. 
The research of M.M.\ has been supported 
by NSF grants DMS-1104734 and DMS-1309298.
\end{ack}


\providecommand{\bysame}{\leavevmode\hbox to3em{\hrulefill}\thinspace}

\bibliographystyle{amsplain}

\end{document}